\documentclass[final,leqno]{siamart0516}
\usepackage{a4}

\usepackage{latexsym,amsfonts}

\usepackage{amsmath}
\usepackage{amssymb}
\usepackage[applemac]{inputenc}
\usepackage{graphicx}
\usepackage{color}
\usepackage{tikz}
\usepackage{enumerate} 
\usepackage{a4wide}

\usetikzlibrary{matrix}
\def\argmin{\mathop{\rm arg\, min}}

\def\F{{\mathcal F}}
\def\bbF{{\mathbb F}}

\def\P{{\cal P}}

\def\R{{\mathbb R}}

\def\X{{\mathcal X}}

\def\Y{{\mathcal Y}}

\def\PROB{{\mathbb P}}
\def\EXP{{\mathbb E}}

\def\IND{{\mathbb I}}

\newtheorem{assumption}[theorem]{Assumption}
\newtheorem{remark}[theorem]{Remark}
\newtheorem{example}[theorem]{Example}
\newcommand{\TheTitle}{Causal transport in discrete time and applications} 
\newcommand{\TheAuthors}{J. Backhoff, M. Beiglb\"ock, Y. Lin and A. Zalashko }

\headers{\TheTitle}{\TheAuthors}

\numberwithin{theorem}{section}
\numberwithin{equation}{section}

\title{ {\TheTitle} 
\thanks{ Submitted to the editors DATE.\funding{The first two authors acknowledge  support by the Austrian Science Fund (FWF) under grants P26736 and Y782-N25  as well as the European Research Council (ERC) under grant FA506041. The third author acknowledges  support by the ERC under Advanced grant 321111. The fourth author acknowledges  support by Doktoratskolleg W1245 of the Austrian Science Fund (FWF).}} }

\author{
Julio Backhoff \thanks{Vienna University of Technology, Vienna, Austria (\email{julio.backhoff@tuwien.ac.at}, +435880110575).} 
\and     Mathias Beiglb\"ock \thanks{Vienna University of Technology, Vienna, Austria (\email{mathias.beiglboeck@tuwien.ac.at}, +435880110570).}  %\thanks{To Mio} 
\and Yiqing Lin \thanks{\'Ecole Polytechnique, Paris, France (\email{yiqing.lin@polytechnique.edu}, +33169334600).}%\thanks{text} 
\and Anastasiia Zalashko \thanks{University of Vienna, Vienna, Austria (\email{anastasiia.zalashko@univie.ac.at}, +431427750722).}
 }

\begin{document}

\maketitle
\begin{abstract}
Loosely speaking, causal transport plans are a relaxation of adapted
processes in the same sense as Kantorovich transport plans extend
Monge-type transport maps. The corresponding causal version of the
transport problem has recently been introduced by Lassalle. Working in
a discrete time setup, we establish a dynamic programming principle that links the causal transport problem to the transport problem
for general costs recently considered by Gozlan et al. Based on this recursive principle, we give conditions under which the celebrated Knothe-Rosenblatt rearrangement can be
viewed as a causal analogue to the Brenier map.
Moreover, these considerations provide transport-information inequalities for the nested distance between stochastic processes pioneered by Pflug and Pichler, and so serve to gauge the discrepancy between stochastic programs driven by different noise distributions.  
\end{abstract}

%\begin{Keywords}  


\section{Introduction}
In this article we consider the optimal transport problem between two discrete-time stochastic processes under the so-called causality constraint, highlighted recently by the work of Lassalle in \cite{Lassalle2} in a more general setting. A transport plan between two processes is said to be \textit{causal} if, from an observed trajectory of the first process, the ``mass'' can be split at each moment of time into the second process only based on the information available up to that time.  It is illustrative to think of the deterministic case (i.e.\ when there is no splitting of mass); such a causal plan is then an actual mapping which is further adapted, and so the relationship between causal plans and adapted processes is the same as between classical transport plans (Kantorovich) and transport maps (Monge). 

The idea of imposing a ``causality'' constraint on a transport plan between laws of processes %is actually not new. From the earliest references on the topic, this definition can be found in the lines of the proof of 
 seems to go back to the Yamada-Watanabe criterion for stochastic differential equations \cite{YW}. Under the name ``compatibility'' the same type of constraint was introduced by Kurtz \cite{Kurtz}. %\comment{Referee 2:change Kurz into Kurtz. Done}
In this article we will also link these objects to the notion of \emph{nested distance}, whose systematic investigation was initiated by Pflug \cite{Pflug} and Pflug--Pichler \cite{PflugPichler, PflugPichlerbook,PflugPichlergeneration}, and had a precursor in the ``Markov-constructions''  studied by R\"{u}schendorf \cite{Rueschendorf}. Roughly, the nested distance is defined through a problem of optimal transport over plans which are \emph{bicausal}, this notion being the symmetrized analogue of causality. Interestingly, \cite{Rueschendorf} and \cite{PflugPichlerbook} established a recursive formulation for the problem, and \cite{PflugPichler,PflugPichlerbook} further obtained a dual formulation for the nested distance. Moreover, Pflug--Pichler \cite{PflugPichler} applied these considerations to the practical problem of reducing the complexity of multistage stochastic programs, by showing that the difference between the optimal value of a program w.r.t.\ two different noise distributions is dominated by the nested distance between them. We refer to the books \cite{ruszczynski2003stochastic,ShapiroBook,PflugPichlerbook} for a detailed account on stochastic programming.

A systematic treatment and use of  causality as an interesting property of abstract transport plans and their associated optimal transport problems was first made by  Lassalle in \cite{Lassalle} in the general context of Polish spaces (then updated in \cite{Lassalle2}). As an application the author considers the Wiener space setting of the problem and establishes that weak solutions to Brownian-motion-driven stochastic differential equation can be conceived as causal transport plans between the Wiener measure and a target measure, and  finds that such plans are automatically bicausal and optimal for a Cameron-Martin-type cost. He then explores functional inequalities in Wiener space (first obtained by \cite{FeyelUstunel}) by means of this method. We stress that the main motivation for our article comes from this connection with stochastic analysis, and our goal is to deepen the understanding of the causal transport problem by looking at the discrete-time setup; the continuous-time counterpart/extension of our results is a work in progress. This motivation implies that for some results we are content with assuming independence of marginals or of increments for some of the processes we look at. It also means that we often seek to show the robustness of some of the particular phenomena obtained by Lassalle: for instance by studying when the causal and bicausal problems coincide/differ, or by showing that functional inequalities are also prevalent in our setting.

The core subject of our article is the causal optimal transport problem, which consists in finding the cheapest causal transport plan from a given source measure (process) to a target one, with respect to a certain cost function on the product space. Since causality can be easily characterized as a linear constraint on transports, we can embed this problem in the class of optimal transport problems under (infinitely many) linear constraints, as considered by Zaev \cite{Zaev} and \cite{BGriessler}; this line of reasoning has  already been applied in the literature, for example in the development of martingale optimal transport (see \cite{HobsonNeuberger,BHP,GHT,DolinskySoner}). In this way, we obtain conditions for the existence of an optimal causal transport and identify a dual formulation for the problem, further establishing no-duality-gap; this is the content of our Theorem \ref{teo:duality}. By studying the conditional distributions of causal transports, we are able to tackle many instances of the causal optimal transport problem by means of a recursion, which we call the dynamic programming principle (DPP in short); see Theorem \ref{thm:DPPcausal}. The appeal of these recursions is that instead of one ``multi-dimensional'' transport problem over causal plans, we obtain recursively several ``one-dimensional'' problems, each one of them a ``general'' (i.e.\ non-linear) transport problem as introduced recently in \cite{weaktransport,weaktransport2}. In this way, we reduce the dimensionality of the problem at the expense of introducing non-linearities.

In Theorem \ref{main theorem} we establish that the Knothe-Rosenblatt rearrangement \cite{knotheConvexBody} (also known as multidimensional quantile transform) %\comment{JB: I deleted the name ``inc. triang. trans.''} 
is causal optimal for the squared euclidean distance (or more generally for convex and separable {cost functions}, in the sense of \eqref{separable cost} below) %\comment{Referee 2: add word ``cost''. Done }
if the source measure is of product type. This  setting is relevant e.g.\ in the theory of functional equalities, and such result can be extended to the case when the source measure has independent increments and the cost is suitably modified (see Corollary \ref{coro KR}) which is a set-up relevant to stochastic analysis. The key here is to first identify the mentioned rearrangement as a bicausal optimizer, which we do in \mbox{Proposition \ref{main result bc}} generalizing a corresponding result in \cite{Rueschendorf}, and then to prove that under the given assumptions the values of the causal and bicausal problems coincide.
 If further the source measure has absolutely continuous marginals, then the Knothe-Rosenblatt rearrangement is of Monge-type. Hence the Knothe-Rosenblatt rearrangement can be viewed as a causal version of the classical Brenier map. In our opinion this result adds to the appeal of this rearrangement, which is in any way widely used in analysis, statistics, and operations research in the context of scenario generation.

We also further the understanding of the relation between nested distance and  multistage stochastic programming established in 
 \cite{Pflug, PflugPichler, PflugPichlerbook}.
%We are also able to further the aforementioned insight in \cite{Pflug, PflugPichler, PflugPichlerbook} relating their nested distance to multistage stochastic programming. 
First we show that many stochastic programs are concave/convex along what we define as lexicographic-displacement interpolations, in analogy to the concepts of displacement interpolation and displacement convex functionals in classical optimal transport theory (see \cite[chapter 5]{Villani} or McCann \cite{McCann_displacement}), where the role of Brenier's map is taken by the Knothe-Rosenblatt rearrangement. Moreover, we give conditions under which the nested distance of ``order one'', which gauges the discrepancy between stochastic programs driven by different noise distributions and a common Lipschitz-cost criterion, can itself be assessed by the square root of the relative entropy between such processes. In other words, we establish a transport-information inequality for this nested distance of ``order one''. This means that the discrepancy between such stochastic programs can be simply gauged by an entropy, which is easier to compute in practice than the nested distance itself. We shall also have occasion to further highlight the connection between causality and functional inequalities when, in Section \ref{sec Talagrand}, we establish Talagrand's celebrated $\mathcal{T}_2$ inequality (see \cite{Talagrand}) for the standard Gaussian measure by interpreting the author's ``tensorization/inductive trick'' as an instance of our recursions; we refer the reader to \cite{LedouxBook,GozlanLeonardSurvey} for an account on functional/geometric inequalities and the related concept of concentration of measure.

The article is organized as follows. In Section \ref{Main results} we introduce the setting and collect our main results; Theorem \ref{teo:duality} on the attainability and duality for the causal transport problem (established in Section \ref{Duality}), Theorem \ref{thm:DPPcausal} on the recursive formulation of the problem (what we call the DPP, whose proof is given in Section \ref{DPP causal case}), Theorem \ref{main theorem} on the identification of the  Knothe-Rosenblatt rearrangement as a causal optimizer (established in Section \ref{Bicausal case}), and finally Theorem \ref{thm: funct nested} on the bicausal transport information inequality which is further explored in Section \ref{lex_section} along with other connections to stochastic programming. In Section \ref{Sec Examples} we present some counterexamples cited throughout the paper.\\

\textbf{Notation:} For a product of sets $\X \times \Y$ we denote by $p^1,p^2$ the projection onto the first resp.\ second coordinate. The pushforward of a measure $\gamma$ by a map $M$ is denoted $M_*\gamma$. We denote by $\gamma^x,\gamma^y$ the regular kernels of a measure $\gamma$ on $\X \times \Y$ w.r.t.\ its first and second coordinate respectively. {Thus $\int f(y)\gamma^x(dy)$ gives a version of the conditional expectation of $f(y)$ given $x$ under measure $\gamma$, sometimes also denoted $E^\gamma[f(Y)|X=x]$ in the literature, and so forth}. %\comment{Referee 2:explain link with cond. expectations. Done?}.
 Analogous notation extends to products of more than two spaces. On $\R^N\times \R^N$ we denote by $(x_1,\dots,x_N)$ the first half and $(y_1,\dots,y_N)$ the second half of the coordinates, and we convene that for $\gamma$ a probability in $\R^N\times\R^N$ (respect.\ $\eta$ on $\R^N$), $\gamma^{x_1,\dots,x_t,y_1,\dots,y_t}$ (respect.\ $\eta^{x_1,\dots,x_t}$) denotes the two-dimensional measure on $(x_{t+1},y_{t+1})$ (respect.\ one-dimensional measure on $x_{t+1}$) given by regular disintegration of $\gamma$ w.r.t.\ $(x_1,\dots,x_t,y_1,\dots,y_t)$ (respect.\ $\eta$ w.r.t.\ $(x_1,\dots,x_t)$). Also, a statement like ``for $\gamma$-a.e.\ $x_1,\dots,x_t,y_1,\dots,y_t$'' or ``for $\eta$-a.e.\ $x_1,\dots,x_t$'' is meant to denote respectively ``almost-everywhere'' with respect to the projections of $\gamma$ onto $x_1,\dots,x_t,y_1,\dots,y_t$ or $\eta$ onto $x_1,\dots,x_t$. {Throughout $C_b(\X)$ stands for the space of continuous, real-valued, bounded functions on $\X$. If $f$ and $g$ are real-valued functions on $\X$ resp.\ $\Y$, we denote $f\oplus g(x,y):=f(x)+g(y)$.}

\section{Main results}\label{Main results}
Let $\X$ and $\Y$ be closed subsets of $\R^N$ and take $\F^{\X}$ and $\F^{\Y}$ the filtrations generated by the coordinate processes (i.e.\ $\F^{X}_t$ is the smallest $\sigma$-algebra s.t.\ $x\in \X\mapsto (x_1,\dots,x_t)\in \R^t$ is measurable, and so forth). The probability measures on the product space $\X \times \Y$ with marginals $\mu, \nu$ correspond to all possible transport plans between the given marginals. Denote this set $$\textstyle\Pi (\mu, \nu) = \left\lbrace \gamma \in \mathcal{P}( \X \times \Y)  \text{ with marginals } \mu \text{ and }\nu \right\rbrace .$$ We will often consider pairs of random variables $(X,Y)$ defined on some probability space $\left( \Omega, \PROB\right) $ and taking values in resp.\ $\X$ and $\Y$, and refer to them as transport plans as well. Any property on $(X,Y)$ should then be understood as a property on $(X,Y)_*\PROB$. 

In the following we assume w.l.o.g.\ that $\X=supp(\mu)$ and $\Y=supp(\nu)$, whenever dealing with transport problems between $\mu$ and $\nu$. With some abuse of notation we will often write $\R^N$, the reader keeping in mind we mean $\X$ or $\Y$. 
%\comment{Referee 1: stress more that you can take general state space. Done}
{
\begin{remark}Throughout this work most of the results would still hold for $S$-valued discrete-time  stochastic processes in $N$-steps, with $S$ a Polish space. That is, we could take $\X=\Y=S^N$. Moreover, we could also take products of different Polish spaces. These can be interpreted as the set of trajectories of discrete time stochastic process taking values in rather arbitrary spaces, for instance $\mathbb{R}^d$. It is mostly for the sake of familiarity that we shall take $S=\R$ throughout the whole article. 
\end{remark}
}  For simplicity, for us being measurable with respect to a sigma algebra means to be equal to a correspondingly measurable function modulo a null set w.r.t.\ the measure unequivocally relevant to the given context.

\begin{definition}\label{def:causality}
A transport plan $\gamma \in \Pi(\mu,\nu)$ is called\footnote{Lassalle  (\cite{Lassalle2}) introduced more technical definitions. The one here is enough for us. {The concept of ``causality'' need not be taken literally here (for instance the independent coupling being causal). Perhaps it helps to think of it as a generalization of adaptedness (i.e.\ non-anticipativity).}}
 causal (between $\mu$ and $\nu$) if for any \linebreak \mbox{ $t\in \{1,\dots,N\}$} and $ B \in {\F^{\Y}_t}$, 
%$ B \in \overline{\F^{\Y}_t}^{\nu}$, 
the mapping $x \in \X \to \gamma^x(B)$ is ${\F^{\X}_t}$-measurable.
% $\overline{\F^{\X}_t}^{\mu}$
The set of all such plans will be denoted $$\textstyle{\Pi_{c}(\mu, \nu)}.$$ 
\end{definition}
Analogously, we will be interested in transport plans that are ``causal in both directions'', or bicausal in our terminology. The set of all such plans is explicitly given by  $$\textstyle\Pi_{bc}(\mu, \nu) = \left\lbrace \gamma \in \Pi_c(\mu, \nu) \text{ s.t. } e_*\gamma \in \Pi_c(\nu, \mu)\right\rbrace ,$$
where $e(x,y) = (y,x)$.
As in the usual optimal transport problem, the set of all causal plans $\Pi_c(\mu,\nu)$, as well as $\Pi_{bc}(\mu, \nu)$, are always non-empty because $\textstyle\mu \otimes \nu\in\Pi_{bc}(\mu, \nu)$. Further, as in the classical setting, we shall consider the case of Borel measurable transformations $T:\X\to \Y$ satisfying $T_*\mu=\nu$, so that in particular $\gamma^T:= (id\times T)_*\mu$ belongs to $\Pi (\mu, \nu)$, and call them (Monge) transport maps. Transport maps are termed (bi)causal if the associated $\gamma^T$ is so.\\

The condition for a transport plan to be causal, written in terms of stochastic processes, looks as follows: for all $t={1,\dots, N}$ and $\textstyle B_t \in \F_t^{\Y}$,
\begin{align*}\textstyle
 \PROB \big( (Y_1,  \dots, Y_t)\in B_{t} \mid X_1,  \dots, X_{N} \big) &= \PROB \big( (Y_1,  \dots, Y_t)\in B_{t} \mid X_1, \dots X_t\big).
\end{align*}
Heuristically this reads as ``given the past of $X$, the past of $Y$ and the future of $X$ are independent''. This is perhaps best interpreted by the following equivalent formulation (see e.g.\ \cite[Proposition 6.13]{Kallbook}):
$$\textstyle Y_t=F_t(X_1,\dots,X_t,U_t)\,\,\,\,\,\,\forall t \in \{1,\dots,N\},$$
%
%\comment{Referee 1: more on the $U_t$'s. Done (JB:I added red text and suggest further explanation only on the reply)}
for some measurable functions $F_t$ and where each $U_t$ is a uniform random variable independent of $X_1,\dots,X_N${; however, the $U_t$'s need not be independent of each other}.
\begin{remark}
Clearly a transport map $T$ is causal if and only if it is adapted, in the sense that there exist Borel-measurable $T^t:\R^t\to\R$ such that for $\mu$-a.e.\ $(x_1,\dots,x_N)$: \vspace{-6pt}
$$\textstyle T(x_1,\dots,x_N)\,\,=\,\, (T^1(x_1),T^2(x_1,x_2),\dots,T^N(x_1,\dots,x_N)).$$
\end{remark}

\vspace{-6pt}The following proposition allows us on the one hand to characterize the causal transport plans using the successive  disintegrations of measures on a product space. On the other hand, it shows that causality can be seen as a linear constraint on measures on the product space, stated in terms of a special class of test functions or via discrete stochastic integrals.
\begin{proposition}\label{prop:equiv charact}
The following statements are equivalent:
\begin{enumerate}
\item  $\gamma$ is a causal transport plan on $\X \times\Y$ between the measures $\mu$ and $\nu$.
\item Decomposing $\gamma$ in terms of successive regular kernels
\vspace{-6pt}\begin{equation}\label{successivedisintegration}\textstyle
\resizebox{.81\hsize}{!}{$
\gamma(dx_1,\dots,dx_N,dy_1,\dots,dy_N)=\bar{\gamma}(dx_1,dy_1)\gamma^{x_1,y_1}(dx_2,dy_2)\dots \gamma^{x_1,\dots,x_{N-1},y_1,\dots,y_{N-1}}(dx_N,dy_N), $}
\end{equation}
then $\textstyle\bar{\gamma}\in \Pi(p^1_*\mu,p^1_*\nu)$ and for $t<N$ and $\gamma$-almost all $x_1,\dots,x_{t},y_1,\dots,y_{t}$  
\vspace{-6pt}\begin{equation}\label{eq: marg mu}\textstyle
p^1_*\gamma^{x_1,\dots,x_{t},y_1,\dots,y_{t}}{(dx_{t+1})}\,\,\, \,\,=\,\,\,\,\, \mu^{x_1,\dots,x_{t}}{(dx_{t+1})},
\end{equation}%\comment{Referee 2:add $dx_{t+1}$ in \eqref{eq: marg mu}. Done.}
and for $\nu$-almost all $y_1,\dots,y_{t}$
\setlength{\belowdisplayskip}{0pt} \setlength{\belowdisplayshortskip}{0pt}
\setlength{\abovedisplayskip}{0pt} \setlength{\abovedisplayshortskip}{0pt}
\begin{equation}\label{eq: marg nu}\textstyle
\gamma^{y_1,\dots,y_{t}}(dy_{t+1})=\nu^{y_1,\dots,y_{t}}(dy_{t+1}).
\end{equation}
%\comment{Referee 1:define $C_b$. Done (Notation paragraph)}
\item $\gamma \in \Pi (\mu, \nu)$ and for all $ t \in\{ 1, \dots, N\}$, $h_t \in C_b(\R^t)$ and $g_t\in C_b(\R^N)$ we have 
%
%\comment{\color{blue}I have introduced the number for the eq 2.4, but it doesn't look nice.}
\setlength{\belowdisplayskip}{0pt} \setlength{\belowdisplayshortskip}{0pt}
\setlength{\abovedisplayskip}{0pt} \setlength{\abovedisplayshortskip}{0pt}
\begin{multline}\label{eq:test caus}\textstyle
\int h_t(y_1,\dots,y_t)\left\{ \,\,g_t(x_1,\dots,x_{N})\,\, -\right.  \\\textstyle \int  \left. g_t(x_1,\dots,x_t,{x}_{t+1},\dots,{x}_N)\mu^{x_1,\dots,x_t}(d{x}_{t+1},\dots,d{x}_{N})\,\, \right\}d\gamma =0. 
\end{multline}

%
%\item $\gamma \in \Pi (\mu, \nu)$ and every $(\F^{\X}\otimes \{\emptyset,\Y\},\gamma)$-martingale is a $(\F^{\X}\otimes \F^{\Y},\gamma)$-martingale. 
%
%\comment{JB: what about continuity of M?}
\item $\gamma \in \Pi (\mu, \nu)$ and for every bounded continuous $\F^{\Y}$-adapted process $H$ and each bounded $\F^{\X}$-martingale $M$ we have $$\textstyle\int \sum_{t<N} H_t(y_1,\dots,y_t)\left [M_{t+1}(x_1,\dots,x_{t+1}) -  M_t(x_1,\dots,x_t)\right ] d\gamma= 0.$$
\end{enumerate}
\end{proposition}
The proof of the above result is given in the next section. Notice that \eqref{eq: marg nu} is equivalent to $$\textstyle\int \gamma^{x_1,\dots,x_{t},y_1,\dots,y_{t}}(\R,dy_{t+1})\gamma^{y_1,\dots,y_{t}}(dx_1,\dots,dx_{t}) \,\,\, \,\,=\,\,\, \,\,\nu^{y_1,\dots,y_{t}}(dy_{t+1}),$$ which is more convenient for the derivation of the dynamic programming principle to come.

%\comment{Referee 1: boundedness of $c$. Done (footnote)}
We now introduce our main optimization problem, the causal optimal transport problem: given some Borel cost\footnote{{In the following, we shall usually assume $c$ to be bounded from below; in principle it would also suffice that $c(x,y)\geq a(x)+b(y)$ for $a\in L^1(\mu),b\in L^1(\nu)$.}} function $c$ defined on $\R^N\times\R^N$ and the probability measures $\mu,\nu$, find the minimal cost at which they can be coupled in a causal way, i.e.\ consider
\begin{align}\textstyle
\inf\limits_{\gamma \in \Pi_c(\mu,\nu)} \int c d\gamma .\tag{Pc}\label{Pc}
\end{align}
Minimizing over the set $\Pi_{bc}(\mu,\nu)$  defines the bicausal optimal transport problem
\begin{align}\textstyle
\inf\limits_{\gamma \in \Pi_{bc}(\mu,\nu)} \int c d\gamma. \tag{Pbc}\label{Pbc}
\end{align}
These should be compared to the classical problem of optimal transport in which minimization is done over $\Pi (\mu, \nu)$.
%
%The following definition shall prove useful: 
%
%\begin{definition}
%The measure $\mu$ is successively weakly continuous, if for each $t<N$, there is a version of the regular conditional kernel of $\mu$ w.r.t.\ its first $t$ variables such that 
%$$(x_1,\dots,x_t)\in supp (\mu)\cap \R^t\mapsto \mu^{x_1,\dots,x_t}(d\bar{x}_{t+1},\dots,d\bar{x}_N )\in \P(\R^{N-t}),$$
%is continuous w.r.t.\ the weak topology in the range and the relative topology in the domain.
%\end{definition}
%
Let us introduce an assumption, which eases the proof of Theorem \ref{teo:duality}:
\begin{assumption}\label{as:weakcont}
The measure $\mu$ is \textit{successively weakly continuous} in the sense that for each $t<N$, there is a version of the regular conditional kernel of $\mu$ w.r.t.\ its first $t$ variables s.t.
$$\textstyle(x_1,\dots,x_t)\in supp (\mu)\cap \R^t\mapsto \mu^{x_1,\dots,x_t}(d{x}_{t+1},\dots,d{x}_N )\in \mathcal{P}(\R^{N-t}),$$
is continuous w.r.t.\ the weak topology in the range and the relative topology in the domain.
\end{assumption}
Let us observe that if $supp(\mu)$ contains no accumulation points, as in the random walk/event tree setting, Assumption \ref{as:weakcont} is vacuously fulfilled. On the other extreme, there are many discrete-time processes with full support satisfying it, e.g.\ the Gaussian %
%\comment{Referee 1:remove ``automatically''. Done} 
case. Let us also notice that $\mu^{x_1,\dots,x_{N-1}}$ is a univariate measure, and more generally we will commonly write $\mu^{x_1,\dots,x_t}(d{x}_{t+1},\dots,d{x}_{t+k})$ for the measure %\comment{Referee 1: remove bars. Done}
$\mu^{x_1,\dots,x_t}(d{x}_{t+1},\dots, dx_{t+k},\R^{N-t-k})$ and similarly $\mu(dx_1,\dots,dx_t)$ for the projection of $\mu$ into the first $t$-marginals.
The following sets of test functions will be instrumental for the dual formulation:%\comment{Do we need to take the linear space generated by this set?}
\begin{equation}\label{set of test functions}\textstyle
\resizebox{.85\hsize}{!}{$
\bbF  := 
\left\{
\begin{array}{c}
 F:\R^N\times\R^N\to\R \mbox{ s.t. } F(x_1,\dots,x_N,y_1,\dots,y_N) =  \\
\sum\limits_{t<N} h_t(y_1,\dots,y_t)\left [g_t(x_1,\dots,x_{N}) - \int g_t(x_1,\dots,x_t,{x}_{t+1},\dots, {x}_N)\mu^{x_1,\dots,x_t}(d{x}_{t+1},\dots,d{x}_N) \right ] , \\
\mbox{with}\,\, h_t\in C_b(\R^t),g_t\in C_b(\R^{N}) \,\,\mbox{for all }t<N  
\end{array}
\right\},$} 
\end{equation}

\begin{equation}\label{set of martingale test functions}\textstyle
\resizebox{.6\hsize}{!}{
$ \mathbb{S}  := 
\left\{
\begin{array}{c}
 S:\R^N\times\R^N\to\R \mbox{ s.t. } S(x_1,\dots,x_N,y_1,\dots,y_N) =  \\
\sum\limits_{t<N}  H_t(y_1,\dots,y_t)\left [M_{t+1}(x_1,\dots,x_{t+1}) -  M_t(x_1,\dots,x_t)\right ], \\
\mbox{with}\,\, H_t,M_t\in C_b(\R^{t}) \,\,\mbox{for all }t<N, \mbox{ and with } M \mbox{ a martingale}  
\end{array}
\right\}.
$} 
\end{equation}
All in all we are ready to present the basic primal attainability/no-duality-gap result: 
%
%\comment{Referee 1: explain $\oplus$-notation. Done (in notation paragraph)}
\begin{theorem}\label{teo:duality}
Suppose that $c:\R^N\times\R^N\to\R \cup \{+\infty\}$ is lower semicontinuous and bounded from below, and that Assumption \ref{as:weakcont} holds. Then there is no duality gap
\begin{align*}\textstyle
\inf\limits_{\gamma \in \Pi_c(\mu,\nu)} \int c d\gamma 
\,\,\, = \,\,\, \sup\limits_{\substack{\Phi,\Psi\in C_b(\R^N),  F\in\bbF \\\Phi  \oplus \Psi  \leq c  +F}}  \left[ \int \Phi d\mu + \int \Psi d\nu\right]
\,\,\, = \,\,\, \sup\limits_{\substack{\Phi,\Psi\in C_b(\R^N),  S\in\mathbb{S} \\ \Phi  \oplus \Psi  \leq c  +S}}  \left[ \int \Phi d\mu + \int \Psi d\nu\right],
\end{align*}
and the infimum on the l.h.s.\ (i.e.\ \eqref{Pc}) is attained. 
\end{theorem}
We observe that Assumption \ref{as:weakcont} allows us to test against continuous bounded functions, instead of just bounded Borel, which is necessary for the simple proof of the previous theorem given in Section \ref{Duality}. However, this assumption can be lifted at the price of losing simplicity (see \cite{Lassalle2}, which was written concurrently), and we choose not to prove the most general statement as our interest lies in other aspects of the problem. It is also easy to see that the dual problem can be reduced to both:
$$ \textstyle\sup\limits_{\substack{\Psi\in C_b(\R^N), F\in\bbF , \Psi  \leq c  +F}}   \int \Psi d\nu \,\,\,\, \mbox{ and }\,\,\,\, \sup\limits_{\substack{\Psi\in C_b(\R^N), S\in\mathbb{S} , \Psi  \leq c  +S}}   \int \Psi d\nu . $$
The analogue of this theorem for bicausal transport plans is given in the next section, and was first obtained in \cite[Theorem 7.2]{PflugPichler}

Going back to Proposition \ref{prop:equiv charact}, the importance of decomposition (\ref{successivedisintegration})  lies in the fact that it suggests that the causal optimal transport problem can be solved recursively if the starting measure $\mu$ is Markovian and the cost function has a ``semiseparable'' structure:
%
% We are of the opinion, however, that both usual multidimensional transport and general causal transport cannot be generally solved in such a way (in any meaningful, tractable fashion).
% Nevertheless, the following result
%establishes that the problem (\ref{Pc}) can indeed be solved recursively at least when the initial measure $\mu$ is Markov, leading us to a dynamic programming principle (DPP). 

%Me must first introduce some necessary notation, handy for Section \ref{Dynamic program in the bicausal case} in any case:
%%
%given a measure $m_1(da)$ and kernels $a\mapsto m_2^a(db,dc)$ and $a\mapsto m_3^a(db)$ we denote the \textit{mixings}:
%\begin{align*}
%m_1 \odot m_2(db,dc)&:=\, \int_a m_1(da)m_2^a(db,dc),\\
%m_1 \odot m_3(db)&:=\, \int_a m_1(da)m_3^a(db),
%\end{align*}
%%
%and observe we have used reversed integral notation, as we will also do often next. Finally we remark that $(m_1\odot m_2)^c$ stands for the kernel of $m_1\odot m_2$ w.r.t.\ the second variable.  
%
%\comment{Referee 1: explain DPP. Done}
%\comment{Referee 1: use $\gamma_t$ instead of $\gamma$. Done}
%\comment{Referee 1: explain role of $m$.}
%\comment{Referee 1:Isn't $V$ linear in $m$?. {\color{blue}JB: to address this and the previous concern, I added an ``example'' below in blue.}}
\begin{theorem}[{Dynamic Programming Principle (DPP)} for causal plans]\label{thm:DPPcausal}
Let $\mu, \nu \in \mathcal{P}(\R^N)$, suppose that $\mu$ is a Markov measure and that the cost is semiseparable in the sense that for non-negative l.s.c.\ functions $c_t$ we have \mbox{$\textstyle c=\sum\limits_t c_t(x_t,y_1,\dots,y_t)$.} \\
Then, starting from $V^c_N:=0$ and defining recursively for $t=N,\dots, 2$:
\begin{multline}\textstyle
V^c_{t-1}(y_1,\dots,y_{t-1};m(dx_{t-1}))\,\,:= \\\textstyle \inf\limits_{\gamma_t\in \Pi\bigl(\int_{x_{t-1}}m(dx_{t-1})\mu^{x_{t-1}}(dx_t)\, ,\, \nu^{y_1,\dots,y_{t-1}}(dy_t)\bigr)} \int \gamma_t(dx_t,dy_t)\Big\{c_t(x_t,y_1,\dots,y_t)  \\
  + V^c_{t}(y_1,\dots,y_t;\gamma_t^{y_t}(dx_{t}))\Big\},\label{weak-trans}
\end{multline}
%
%we have that the integrals above are well-defined and moreover
%
and
$$\textstyle V^c_0:= \inf\limits_{\gamma_1\in \Pi(p^1_*\mu,p^1_*\nu )} \int \gamma_1(dx_1,dy_1) \{c_1(x_1,y_1) + V^c_1(y_1;\gamma_1^{y_1}(dx_1))\}, $$
the integrals above are well-defined and $V_0^c={value\eqref{Pc}}$, i.e.\ the recursion determines \eqref{Pc}. Furthermore, if Assumption \ref{as:weakcont} holds, there exists a causal optimizer $\tilde{\gamma}$ with the additional
property that for all $t$: $$\textstyle\tilde{\gamma}^{x_1,\dots,x_t,y_1,\dots,y_t}(dx_{t+1},dy_{t+1})= \tilde{\gamma}^{x_t,y_1,\dots,y_t}(dx_{t+1},dy_{t+1}) ,$$ and each of the one-step optimization problems in \eqref{weak-trans} corresponds to a \textit{general transport problem} of \cite{weaktransport} which is convex and l.s.c.\ in the kernel (i.e.\ in $\gamma_t^{y_t}$, for fixed $y_1,\dots,y_{t-1}$).
\end{theorem}

The very last line of the previous result means concretely that 
\begin{multline}\textstyle
V^c_{t-1}(y_1,\dots,y_{t-1};m(dx_{t-1}))\,\,= \\ \textstyle\inf\limits_{\substack{y_t\mapsto \gamma_t^{y_t} \,\, s.t. \\ \int_{y_t} \nu^{y_1,\dots,y_{t-1}}(dy_{t})\gamma^{y_t}(dx_t)\, =\, \int_{x_{t-1}}m(dx_{t-1})\mu^{x_{t-1}}(dx_t)     }}\int_{y_t} \nu^{y_1,\dots,y_{t-1}}(dy_{t}) \Big\{ \int_{x_t} \gamma_t^{y_t}(dx_t)c_t(x_t,y_1,\dots,y_t)   \\
 + V^c_{t}(y_1,\dots,y_{t-1},y_t;\gamma_t^{y_t}(dx_{t}))\Big\},\label{weak-trans-2}
\end{multline}
and that the function in curly brackets in the r.h.s.\ is convex and l.s.c.\ in $\gamma_t^{y_t}$ (ceteris paribus). The same function in brackets is only jointly universally measurable (actually lower semianalytic) as far as we can say, regardless of Assumption \ref{as:weakcont}. If on the other hand \mbox{Assumption \ref{as:biweakcont}} holds, then this function is jointly l.s.c. 
In \cite{weaktransport} problems of this kind have been \mbox{analysed} and their duality theory has been established; we provide the dual formulation of \mbox{$V^c_{t-1}$ in \eqref{dual DPP}}. { To fully grasp the nature of the DPP one should better look at an example; this requires the use of some extra notation which will be  also useful also for Theorem \ref{main theorem} below. We denote by $$\textstyle F_{\eta}(\cdot):= \nu((-\infty,\cdot]),$$ the usual cumulative distribution function of a probability measure $\eta$ on the line, and by  $F^{-1}_{\eta}(u)$ its left-continuous generalized inverse, i.e.\ $$\textstyle F^{-1}_{\eta}(u) = \inf\left\lbrace y: F_{\eta} (y) \geq u \right\rbrace .$$

 %we are about to introduce, so we ask the reader to wait until Example \ref{Ex DPP}. }% In Corollary \ref{coro:wass one} we will have a chance to use recursion \eqref{weak-trans} in a very specific situation, providing a first application of it.\\

%As promised in Section \ref{Main results}, we now illustrate the difficulty of the DPP with an example. 
\begin{example}\label{Ex DPP} Take $N=2$ and $c=[x_1-y_1]^2+[x_2-y_2]^2$. Using the well-known optimality of the monotone coupling on the line we get:
\begin{align*}
&\textstyle V_1(y_1;\gamma_1^{y_1}(dx_1)):= \inf\limits_{\gamma_2\in\Pi(\int_{x_1}\gamma_1^{y_1}(dx_1)\mu^{x_1},\nu^{y_1})}\int \gamma_2(dx_2,dy_2)[x_2-y_2]^2=\int_0^1\Bigl[F_{\nu^{y_1}}^{-1}(u)- F_{\int_{x_1}\gamma_1^{y_1}(dx_1)\mu^{x_1}}^{-1}(u)\Bigr]^2du,\\
&\textstyle V_0 =  \inf\limits_{\gamma_1\in\Pi(p^1_*\mu,p^1_*\nu)}\Bigl \{ \int_{x_1,y_1} \gamma_1(dx_1,dy_1)[x_1-y_1]^2 + \int_{y_1}\nu(dy_1)\int_0^1\Bigl[F_{\nu^{y_1}}^{-1}(u)- F_{\int_{x_1}\gamma_1^{y_1}(dx_1)\mu^{x_1}}^{-1}(u)\Bigr]^2du\Bigr\}.
\end{align*}
From this the non-linear behaviour of the cost function in $V_0$, in terms of $\gamma_1$, is apparent in its last term.
\end{example}
}

%
%We thus see great non-linearities in the inner infinum. Problems of this kind fit into the framework of so-called ``weak transportation problems'' of \cite{weaktransport}. However our inner infinum, as a function of the measure $\bar{\gamma}^{y_1}$, does not seem to have the amenable structure used in the given reference (e.g.\ convexity w.r.t.\ $\bar{\gamma}^{y_1}$).\\

%A similar, yet much more tractable DPP holds for \eqref{Pbc} even if $\mu$ is not Markov (see Section \ref{Bicausal case}). This will lead to identifying in Theorem \ref{main theorem} the optimizer of \eqref{Pbc} under a slightly more general condition than in \cite[Corollary 2]{Rueschendorf}. 
Our next main result, Theorem \ref{main theorem} (which we prove in Section \ref{Bicausal case}), establishes the equivalence of \eqref{Pc} and \eqref{Pbc} under Condition \ref{independence cond} below, and furthermore, it identifies the causal optimizer of \eqref{Pc} for convex separable costs (as well as clarifying when these are Monge maps), in a setting relevant for future applications.
%The same result can be obtained for the bicausal optimal transport problem:
%%
%
%For better presentation we used the reversed integration notation in the previous theorem, i.e.\ $ \int a(z)b(dz)=  \int db(z)a(z)$ and where $\mu^{x_1,\dots,x_{t}}$ denotes here the one-step disintegration $\mu^{x_1,\dots,x_{t}}(d\bar{x}_{t+1})$.
%
%\comment{Referee 1: do not call KR rearrangement, but coupling. {\color{blue}JB: I explain that some readers could know it by the term ``coupling''}}
 We denote $F_{\eta_1}$ the distribution of $p^1_*\eta$ whenever $\eta$ is a measure in multiple dimensions. The increasing $N$-dimensional \textbf{Knothe-Rosenblatt} rearrangement\footnote{ {The reader might know it by the name \emph{quantile transform} or \emph{Knothe-Rosenblatt coupling}.}} of $\mu$ and $\nu$ is defined as the law of the random vector $(X_1^*,\dots,X_N^*, Y_1^*,\dots,Y_N^*)$ where
%
%\comment{Referee 1: you changed from $t$ to $n$ as running index. Done (now is $t$)}
\begin{align}\label{quantile transforms}\textstyle
& X_1^* = F_{\mu_1}^{-1} (U_1),\hspace{46pt}  Y_1^* = F_{\nu_1}^{-1} (U_1), \hspace{31pt}\mbox{ and inductively }\\ \nonumber
& X_t^* = F_{\mu^{X_1^*,\dots,X_{t-1}^*}}^{-1} (U_t),\quad Y_t^* = F_{\nu^{Y_1^*,\dots,Y_{t-1}^*}}^{-1} (U_t),\,\,\,\, \text{for } t=2,\dots,N,
\end{align}
for $U_1,\dots,U_N$ independent and uniformly distributed random variables on $[0,1]$. Additionally, if $\mu$-a.s.\ all the conditional distributions of $\mu$ are atomless (e.g.\ if $\mu$ has a density), then this rearrangement is induced by the (Monge) map%\footnote{Also called \textit{increasing triangular transform}}
$$\textstyle (x_1,\dots,x_N)\mapsto T(x_1,\dots,x_N):=(T^1(x_1),T^2(x_2;x_1),\dots,T^N(x_N; x_1, \dots. x_{N-1})),$$
where 
\begin{align}\textstyle
T^1(x_1)&:= \,\, F_{\nu_1}^{-1}\circ F_{\mu_1}(x_1), \nonumber\\
T^t(x_t; x_1,\dots,x_{t-1})& :=\,\, F_{\nu^{T^1(x_1),\dots,T^{t-1}(x_{t-1};x_1,\dots,x_{t-2})}}^{-1}\circ F_{\mu^{x_1,\dots,x_{t-1}}}(x_t),\,\,\,\, t\geq 2. \label{KRdef}
\end{align}
% 

%
%\begin{condition}\label{independence cond}
%%The marginals of the source measure $\mu$, denoted as $\mu_1(dx_1),\dots,\mu_N(dx_N)$, are independent.
%The measure $\mu$ has independent marginals, i.e.\ the coordinate functions $x_1, \ldots, x_N$ are independent random variables on $(\R^N, \mu)$, and $\mu(dx_1,\dots, dx_N) = \mu_1(dx_1) \dots \mu_N(dx_N)$.
%\end{condition}
%
%\begin{condition}[\cite{Rueschendorf}] \label{Rueschendorf cond}
%For any $t=1,\dots, N-1$, the distribution functions of $\mu^{x_1,\dots,x_{t}}$ and $\nu^{y_1,\dots,y_{t}}$ are both monotonically nondecreasing (or nonincreasing).
%\end{condition}
%

\begin{theorem}\label{main theorem}
Assume that $c$ is l.s.c.\  bounded from below and has a separable structure
\begin{align}\label{separable cost}\textstyle
c(x_1,\dots,x_{N},y_1,\dots,y_{N})&=\textstyle \sum\limits_{t\leq N} c_t(x_t,y_t).
\end{align}
Further suppose that the starting measure $\mu$ is the product of its marginals, i.e.
\begin{equation}\textstyle
\mu(dx_1,\dots, dx_N) = \mu_1(dx_1) \dots \mu_N(dx_N). \label{independence cond}
\end{equation}
Then the values of \eqref{Pc} and \eqref{Pbc} coincide. If moreover, $c_t(x,y)=c_t(x-y)$ and $c_t$ is convex%\comment{Referee 2: explain abuse of notation in the footnote. Done}
\footnote{{We overload notation here; we mean $c_t(x,y)=h_t(x-y)$ and redefine $c_t$ as $h_t$. This ``univariate'' structure is not indispensable for the result: it can be replaced by a more general Spence-Mirrlees condition.}}, then a solution to \eqref{Pc} is given by the Knothe-Rosenblatt rearrangement
% \eqref{quantile transforms}, which takes the specific form %\comment{MB: It seems ridiculous that we restate this here!}
%
%\begin{align}\label{quantile transformsbc}
%& X_1^* = F_{\mu_1}^{-1} (U_1),\hspace{15pt}  Y_1^* = F_{\nu_1}^{-1} %(U_1), \hspace{41pt}\mbox{ and inductively }\\ \nonumber
%& X_n^* = F_{\mu_n}^{-1} (U_n),\hspace{15pt}  Y_n^* = F_{\nu^{Y_1^*,%\dots,Y_{n-1}^*}}^{-1} (U_n),\quad \text{for } n=2,\dots,N.
%\end{align}
%
Additionally, if each $\mu_i$ is atomless (e.g.\ if they have a density), then this rearrangement is induced by the Monge map \eqref{KRdef}. 
%explicitly given by
% 
%$$T(x_1,\dots,x_N)=(T^1(x_1),\dots,T^N(x_N; x_1, \dots. x_{N-1}))$$
%%
%which is defined as follows: 
%
%\begin{align}
%T^1(x_1)&:= \,\, F_{\nu_1}^{-1}\circ F_{\mu_1}(x_1), \nonumber\\
%T^n(x_n; x_1,\dots,x_{n-1})& :=\,\, F_{\nu^{T^1(x_1),\dots,T^{n-1}%(x_{n-1};x_1,\dots,x_{n-2})}}^{-1}\circ F_{\mu_n}(x_n),\,\,\,\,\, n%%\geq 2. \label{KRdefbc}
%\end{align}
%
%for $n=2,\dots,N$. 
\end{theorem}

%Throughout this work we mostly focus on convex costs of the difference, and as a result we only deal with the increasing Knothe-Rosenblatt rearrangement. Evidently, for concave costs the correct notion would be to take the decreasing rearrangement defined in the obvious way.\\

\begin{corollary}
\label{coro KR}
Assume that $c$ is l.s.c.\  bounded from below and is of the form \begin{align}\label{separable cost increments}\textstyle
c(x_1,\dots,x_{N},y_1,\dots,y_{N})&= \textstyle c_1(x_1,y_1)+ \sum\limits_{1<t\leq N} c_t(x_t-x_{t-1},y_t-y_{t-1}),
\end{align}
and the source measure $\mu$ has independent increments. Then the values of \eqref{Pc} and \eqref{Pbc} coincide.  If moreover, $c_t(a,b)=c_t(a-b)$ and $c_t$ is convex then a solution to \eqref{Pc} is given by the Knothe-Rosenblatt rearrangement \eqref{quantile transforms} and if additionally each $\mu_i$ is atomless  then this rearrangement is induced by the Monge map \eqref{KRdef}.
\end{corollary}

%Let us remark that independence of marginals and separability of the cost are far from being necessary for the causal and bicausal problems to coincide. \comment{JB:I added this comment for the indep. increments case}
%On the other hand, it is trivial to see that if $\mu$ has independent increments instead of independent marginals, and the cost is of the form $$\sum\limits_{t=1}^{N-1} c_t(x_{t+1}-x_t,y_{t+1}-y_t),$$ then again causal and bicausal problems are equivalent. If further these $c_t$'s are all convex, then again a common optimal bicausal solution is the Knotte-Rosenblatt rearrangement (which is again a map under the corresponding assumptions). 

In \cite[Lemma 5]{Lassalle} it was shown that the optimal solution to the causal transport problem in Wiener Space, with Cameron-Martin cost and Wiener measure as source, is bicausal. %As we explain in Section \ref{Sec:projection}, projecting (pushing forward) that problem in Wiener space into finitely many time points, one gets a causal problem in Euclidean space with quadratic-type separable cost and a measure with independent components as a source. Therefore, 
The results above give the causal/bicausal equality and existence as well as characterization of the Monge solution to what can be thought of as ``finite-dimensional projections'' of that problem. We can only guess then that causal/bicausal equality in continuous-time is a result of the Cameron-Martin cost being written in terms of ``speed'' and the source measure having independent increments. In Section \ref{Sec Examples} we give two counterexamples showing that if either independence or separability of the cost is dropped, the causal-bicausal equality may fail. 

One question that the above considerations and results leave open, is to what extent the Knothe-Rosenblatt rearrangement is canonical. This is answered in Section \ref{sec Bogashev}, where we show how it is characterized as the only \emph{increasing} transport (in a precise sense) which is also bicausal. In Section \ref{lex_section} we will further highlight the importance of this rearrangement by viewing it in light of a modern tenant of optimal transport theory. Inspired by the concepts of displacement interpolation/convexity (as in \cite[Chapter 5]{Villani}), we will define a related notion where the Knothe-Rosenblatt rearrangement replaces the role of the Brenier's map, and dub it  lexicographical displacement interpolation/convexity. %We put forward the idea that lexicographical displacement interpolation, by having ``constant speed'' with respect to many nested distances (equivalently bicausal Wasserstein distances), should have a geometric meaning as geodesic curves w.r.t.\ a structure on the space of probability measures which as far as we know has not been discovered yet; so the materialization of this idea is an open problem, and we see this as a relevant step towards an understanding of the geometrical side of the bicausal transport problem. 
We expect that the lexicographical displacement interpolations should have a geometric meaning as geodesic curves, however the materialization of this idea is left as an open problem, which we  consider as a relevant step towards an understanding of the geometrical side of the bicausal/nested transport problem. We only provide indirect evidences in this respect. Still, this also provides an interesting link to stochastic programming and nested distances. Concretely, we show that under convexity and Lipschitz conditions on the cost, stochastic programs are lexicographical displacement concave in analogy to the potential energy in optimal transport (\cite[Chapter 5.2]{Villani}).

%\comment{JB: This is new}
We turn our attention to multistage stochastic programming, and introduce now our last main results. The goal here is to explore and to further the connection between (non-anticipative) multistage stochastic programming and bicausal optimal transport, discovered by Pflug and Pichler \cite{PflugPichler}, from the point of view of geometric/functional inequalities. We start by defining the value function of a stochastic program
\begin{equation}\label{def stoch prog}\textstyle
v(\eta):=\inf_{u_1(),\dots, u_N()} \int H(x_1,\dots,x_N,u_1(x_1),u_2(x_1,x_2),\dots,u_N(x_1,\dots,x_N))\eta(dx_1,\dots,dx_N) ,
\end{equation} 
as a function of the noise distribution $\eta$. Here $H:\R^N\times \R^N \to\R$ is the objective function and the minimization is taken over Borel adapted controls. As established \mbox{in \cite[Theorem 6.1]{PflugPichler},} as soon as $H$ is $1$-Lipschitz in its first argument and convex in the second one, the discrepancy $|v(\mu)-v(\nu)|$ is less than the bicausal distance between $\mu$ and $\nu$ w.r.t.\ the cost $$\textstyle c(x,y):=\|x-y\|_1=\sum_i|x_i-y_i|.$$ This means that if say $\mu$ is a ``complicated law'' (for instance given by a nightmarish tree), then finding a ``simpler law'' $\nu$ close enough in the causal distance sense guarantees that uniformly in the given class of stochastic programs the discrepancy is controlled. Therefore, the practical question is how to efficiently minimize this bicausal distance over a given set of measures. We argue that in many situations a transport-information inequality for such bicausal distances permits to gauge the discrepancy of stochastic programs replacing the computation of such distance by the evaluation of a relative entropy. Notice that in our interpretation, the complicated measure $\mu$ is expected to dominate (in the sense of absolute continuity) the simpler measure $\nu$.  Our result is reminiscent of \cite[Proposition 4.6]{PflugPichlerbook}; crucially the presence of the entropy here stems from our assumption (EXP). 

\begin{theorem}\label{thm: funct nested}
Let $\mu\in \mathcal{P}(\mathbb{R}^N)$ satisfy:
\begin{itemize}
\item[(EXP)] for each $t$ there exist $a_t>0,\lambda_t\in\R$ such that $\textstyle\int e^{a_tx_t^2}\mu^{x_1,\dots ,x_{t-1}}(dx_t)\leq e^{\lambda_t}$ $\mu$-a.s.
%\item[(EXP)] there exists $a>0$ such that $\int e^{a[x_1^2+\cdot+ x_N^2]}\mu(dx_1,\dots ,d x_N)<\infty$
\item[(LIP)] There is $C>0$ such that for all $t\in\{1,\dots,N\}$, and every $1$-Lipschitz function $f:\R \to\R$, the function
$$\textstyle (x_1,\dots,x_{t-1})\mapsto \int f(x_t)\mu^{x_1,\dots,x_{t-1}}(dx_t),$$
is $\mu$-a.s.\ equal to a $C$-Lipschitz function.
\end{itemize}
Then we have the following bicausal transport-information inequality: 
\begin{equation}\label{eq: T1bc}\textstyle
\mbox{for all }\nu\in \mathcal{P}(\mathbb{R}^N) :\,\,\,\,\mathcal{W}_{1,bc}(\mu,\nu):=\inf_{\gamma\in \Pi_{bc}(\mu,\nu)}\int \|x-y\|_1 d\gamma(x,y)\leq K\sqrt{Ent(\nu|\mu)},
\end{equation}
where $Ent(\cdot|\mu)$ denotes the relative entropy with respect to $\mu$ and $$\textstyle K:=\sqrt{2\sum\limits_{j<N}(1+C)^{2j} \frac{(1+\lambda_{N-j})}{a^2_{N-j}}}.$$ In particular, for every ``cost criterion'' $H:\R^N\times \R^N \to\R$ bounded from below, $r$-Lipschitz in its first argument and convex in its second, we have for the corresponding stochastic programs
\begin{equation}\label{eq ent upp}\textstyle
|v(\mu)-v(\nu)| \leq rK\sqrt{Ent(\nu|\mu)},
\end{equation}
with $v(\cdot)$ defined as in \eqref{def stoch prog}.
\end{theorem}

%{\color{red}JB:Under (BLA) we get $K=\sqrt{2\sum\limits_{j=0}^{N-1}(1+C)^{2j} \frac{(1+\lambda_{N-j})}{a^2_{N-j}}}$}

%\comment{Referee 1:comment on the possibility that entropy is infinite. Done (see blue text)}\comment{Referee 1:do not see the relationship with rest of paper {\color{blue}JB: I say finite discrete now}}
\begin{remark} A few observations are in order:
\begin{enumerate}
\item {The entropic upper bounds in \eqref{eq: T1bc}-\eqref{eq ent upp} of course trivialize, becoming $+\infty$, if $\nu\not\ll\mu$. On the other hand, one could expect the r.h.s.\ of \eqref{eq ent upp} to be easier to compute (when finite) than a transport-type quantity. }
\item W.l.o.g.\ the Lipschitz property above is meant with respect to the sum-of-absolute-values norms in the respective spaces. By the tower property $(EXP)$ implies $$\textstyle\int e^{a_1x_1^2+\cdot+ a_Nx_N^2}\mu(dx_1,\dots ,d x_N)\leq e^{\lambda_1+\dots+\lambda_N}<\infty,$$
and in particular every Lipschitz function is $\mu$-integrable.
%\comment{Referee 2: what do you mean with ``discrete case''.{\color{blue} Pending}}
\item We stress that Assumption $(EXP)$ is reasonable in practice; it is automatically satisfied in the finite discrete case, for empirical measures, or when $\mu$ has bounded support. The corresponding non-causal version of \eqref{eq: T1bc}, which can be found in \cite[Theorem 22.10]{Villani_Old} or \cite[Theorem 2.3]{transport_markov} and crucially does not implies ours, actually also needs a condition equivalent to $(EXP)$.
\item Assumption $(LIP)$ implies, by the Kantorovich-Rubinstein Theorem (e.g.\ \cite[Theorem 1.14]{Villani}) that
$$\textstyle\mathcal{W}_1(\mu^{x_1,\dots,x_{t-1}},\mu^{z_1,\dots,z_{t-1}})\leq C\sum_{i<t} |x_i-z_i|,$$
where $\mathcal{W}_1$ is the usual 1-Wasserstein distance on $\mathcal{P}(R)$. If $\mu$ is a Markov law, then the r.h.s.\ here can be taken as  $C |x_{t-1}-z_{t-1}|$. On the other hand, if $\mu$ is a martingale law, the l.h.s.\ is in turn bounded from below by $ |x_{t-1}-z_{t-1}|$.
\item All in all, Assumption $(LIP)$ is the most fundamental for Theorem \ref{thm: funct nested}. It is however clearly satisfied in the discrete case or for empirical measures. In general it is also implied if $\mu$ has independent increments (with $C=1$ then), or if e.g.\ $\mu$ is the law of the solution ${X_1,\dots,X_N}$ of a uniformly Lipschitz random dynamical system of the form $X_{t+1}=R(Z_t,(X_1,\dots,X_t),t)$, where $R$ is Borel in the first argument, Lipschitz in the second one (uniformly w.r.t.\ the first one), and $Z_t$ is independent of $(X_1,\dots,X_t)$.
\end{enumerate}
\end{remark}

The proof of the previous theorem is given in Section \ref{lex_section}. We observe that the constant in \eqref{eq: T1bc} can be very plausibly improved (making it of order $\sqrt{N}$ is probably the best possible). This is done for example in \cite[Theorem 2.5]{transport_markov} for the non-causal but Markov case, and necessitates more care and further assumptions.%For the application we have in mind, the constant is largely irrelevant.\\ %We also mention that we focused in the Lipschitz case above since H\"olderian transport-information inequalities 

% The equality between causal and bicausal problems is achieved due to the independent increments of the Brownian motion that is proved in the second part of the previous theorem (discretizing Brownian motion, we are in the case of the Condition (\ref{independence cond}), i.e. the source measure with independent marginals). \textit{This fact also shows the relevance of Condition (\ref{independence cond}) for our purposes and possible applications}. 

\section{Duality}\label{Duality}
The dual to the causal transport problem is discussed in this section. 
%First, we provide without proof the following easy quantitative test for adaptability. 
%We introduce, for ease of presentation, the function $P^t:\R^N\times\R^N\to\R^2$ defined as the projection w.r.t.\ the variable $(x_t,y_t)$ and recall that $p^t:\R^N\to\R$ denotes the $t$-th projection of a vector.
First, we give the proof of Proposition \ref{prop:equiv charact}.

\begin{proof}[Proof of Proposition \ref{prop:equiv charact}]$ $

STEP 1: Equivalence\ between Points 1 and 3:

Denote $f^h(x_1,\dots,x_N): = \int h_t(y_1,\dots,y_t)  \gamma^{x_1,\dots,x_N}(dy_1,\dots,dy_t)$ with $h_t \in C(\R^t)$. By definition \\ $\gamma \in \Pi_c(\mu,\nu)$ if and only if for all $t\leq N$ and all such $f^h$ we have%\comment{Referee 2: in next three equations $\bar{x}_{t_1}$ should be $\bar{x}_{t+1}$. Done}
\begin{align}\notag\textstyle
f^h(x_1,\dots,x_N) = \int f^h(x_1,\dots,x_t,\bar{x}_{t+1},\dots,\bar{x}_N)  \mu^{x_1,\dots,x_t} (d{{x}_{t+1}}, \dots, d{x}_{N}),
\end{align}
which is equivalent to the following:
\begin{align*}\textstyle
\int g(x_1,\dots,x_N)\left[ f^h(x_1,\dots,x_N) -  \int f^h(x_1,\dots,x_t,{x}_{t+1},\dots,{x}_N)  \mu^{x_1,\dots,x_t}(d{{x}_{t+1}}, \dots, d{x}_{N}) \right] d\mu= 0,
\end{align*}
for every function $g \in C_b(\R^N)$ and for all $t\leq N$. The fact we can take the $g$'s continuous and not merely Borel bounded comes from the fact that $\mu$ is a Borel finite measure on a Polish space. It is easy to see that the previous equation is equivalent to 
\begin{align*}\textstyle
\int f^h(x_1,\dots,x_N)  \left[ g(x_1,\dots,x_N) -\int g(x_1,\dots,x_t,{x}_{t+1},\dots,{x}_N)  \mu^{x_1,\dots,x_t}(d{{x}_{t+1}}, \dots, d{x}_{N})\right ] d\mu= 0.
\end{align*}
Finally, by the tower property of conditional expectations the latter is equivalent to: 
\begin{multline*}\textstyle
\int h_t(y_1,\dots,y_t)\left [ g_t(x_1,\dots,x_{N})-  \int\Big.  g_t(x_1,\dots,x_t,{x}_{t+1},\dots,{x}_N)\mu^{x_1,\dots,x_t}(d{x}_{t+1},\dots,d{x}_{N}) \right ]d\gamma =0. 
\end{multline*}
%\comment{Referee 2: Step 2 is equiv. between points 1 and 2. Done\\
%Referee 2: say more on the induction.{\color{blue} Done (in red)}}
STEP 2: Equivalence between Points 1 and {2}:

Let $\gamma\in\Pi(\mu,\nu)$ be decomposed as in \eqref{successivedisintegration}.
%By induction, one readily sees that causality is equivalent to the conditions (for all $t$).\comment{Referee 2: missing a dot. Done}
%
{
It is causal if and only of for any time $ {t  \leq N}$,
%Causality is equivalent to the following equality between kernels 
$\gamma^{x_1,\dots,x_t,y_1,\dots,y_t}(dx_1,\dots,dx_N) = \gamma^{x_1,\dots,x_t}(dx_1,\dots,dx_N)$
%$ \EXP [\IND_H | \F_t^X, \F_t^Y] = \EXP [\IND_H | \F_t^X] , $
%
%for any 
%$H \in \F^X $ 
(see \cite[Proposition 6.6]{Kallbook}).
Since the  $x$-marginal of $\gamma$ is $\mu$, these facts imply \eqref{eq: marg mu}. On the other hand, the $y$-marginal of $\gamma$ is $\nu$,  so \eqref{eq: marg nu} directly follows. 
%The equality \eqref{eq: marg nu} stems from % disintegrating  the measure $\gamma$ first w.r.t. $(x_1,\dots,x_t,y_t,\dots,y_t)$, then w.r.t. $(y_1,\dots,y_t)$ and 
%the fact the .
%is also satisfied for any $t<N$,
% $\gamma(dx_1,\dots,dx_t,dy_1,\dots,dy_t)$ has marginal $\nu(dy_1,\dots,dy_t)$. And we check this for $t+1$
%%
%\begin{align*}
%& \int h(y_1,\dots, y_{t+1}) \gamma (dx_1,\dots,dx_{t+1},dy_1,\dots,dy_{t+1})  = \\
%& \int h(y_1,\dots, y_{t+1}) \gamma^{x_1,\dots,x_t,y_1,\dots,y_t}(dx_{t+1},dy_{t+1}) \gamma(dx_1,\dots,dx_t,dy_1,\dots,dy_t) =  \\
%& \int h(y_1,\dots, y_{t+1}) \gamma^{x_1,\dots,x_t,y_1,\dots,y_t}(dx_{t+1},dy_{t+1}) \gamma^{y_1,\dots,y_t}(dx_1,\dots,dx_t) \gamma(dy_1,\dots,dy_t)= \\
%& \int h(y_1,\dots, y_{t+1})) \nu^{y_1,\dots,y_t}(dy_{t+1}) \nu(dy_1,\dots,dy_t)= \int h(y_1,\dots, y_{t+1}) \nu (dy_1,\dots,dy_{t+1}).
%\end{align*}
%To see that these conditions 
%\eqref{eq: marg mu} and \eqref{eq: marg nu} 
%imply causality, 
For the converse direction, it is enough to verify \eqref{eq:test caus} for any $t=1, \dots, N-1$. 
%We show that it holds for some intermediate step  $t=i<N-1$. 
%and the rest follows by induction.
Since the functions $h_t$ in \eqref{eq:test caus} depend only on $y_1,\dots,y_t$, the latter can be computed as
\vspace{-10pt}\begin{multline*}\textstyle
%\int  h_i(y_1,\dots,y_i)\left [g_i(x_1,\dots,x_{N}) \right. -  \\  
%\left. \int g_i(x_1,\dots,x_i,{x}_{i+1},\dots, {x}_N)\mu^{x_1,\dots,x_i}(d{x}_{i+1},\dots,d{x}_N) \right ] d \gamma  = \\
\int  h_{t}(y_1,\dots,y_{t}) \left[ g_t(x_1,\dots,x_{N}) - \int g_t(x_1,\dots,x_{t}, {x}_{t+1},\dots,{x}_{N}) \mu^{x_1,\dots,x_{t}}(d{x}_{t+1},\dots, d{x}_N) \right]  \\ 
\mu^{x_1,\dots,x_{N-1}}(dx_N) \dots \mu^{x_1,\dots,x_{t}}(dx_{t+1})\gamma^{x_1,\dots,x_{t-1},y_1,\dots,y_{t-1}}(dx_{t},dy_{t}) \dots \gamma^{x_1,y_1}(dx_2,dy_2) \bar{\gamma}(dx_1,dy_1) ,
\end{multline*}
which is zero as desired because of
%
%From the disintegration property we know that
%
$ \mu^{x_1,\dots,x_{N-1}}(dx_N) \dots \mu^{x_1,\dots,x_{t}}(dx_{t+1}) = \mu^{x_1,\dots,x_{t}}(dx_{t+1}, \dots, dx_N) $
(disintegration property). 
%Hence, causality is equivalent to the conditions \eqref{eq: marg mu} and \eqref{eq: marg nu} (for all $t$).
}
%$$p^1_*\gamma^{x_1,\dots,x_{t},y_1,\dots,y_{t}}\,\,\, \,\,=\,\,\,\,\, \mu^{x_1,\dots,x_{t}}.$$
%Thus for the equivalence of points 1.\ and 2.\ we only need to argue that knowing $\gamma\in\Pi(\mu,\nu)$ implies \eqref{eq: marg nu} and then conversely, that knowing \eqref{eq: marg nu},  \eqref{eq: marg mu} and  $\bar{\gamma}\in \Pi(p^1_*\mu,p^1_*\nu)$ implies $\gamma\in\Pi(\mu,\nu)$. The direct statement is clear, since in the l.h.s.\ of \eqref{eq: marg nu} the dependence on the $x$'s is integrated out, so what is left is $\gamma^{y_1,\dots,y_t}(dy_{t+1})$, which must coincide with a kernel of $\nu$. For the converse statement, the fact that the $x$-marginal of $\gamma$ is $\mu$ follows from \eqref{eq: marg mu} and  $\bar{\gamma}\in \Pi(p^1_*\mu,p^1_*\nu)$, whereas the fact that the $y$-marginal of $\gamma$ is $\nu$ follows from $\bar{\gamma}\in \Pi(p^1_*\mu,p^1_*\nu)$ and \eqref{eq: marg mu} after the above observation on the l.h.s.\ of such equation.

STEP 3: Equivalence between Points 3 and 4:
%Clearly Point 3 implies Point 5, by simply choosing $g_t=M_{t+1}$, $h_t=H_t$ and summing up. Al

Evidently in Point 3 we could have taken $h_t$ and $g_t$ Borel bounded, as STEP 1 suggests. Choosing then $g_t=M_{t+1}$ and $h_t=H_t$ for each $t<N$, and summing up, proves Point 4 from Point 3. Conversely, given {$t$, }$h_t$ and $g_t$ we build $H_s= h_t {\bf{1}}_{s\geq t}$ and $M_s = \int g_t(x_1,\dots,x_s,x_{s+1},\dots,x_N)\mu^{x_1,\dots,x_s}(dx_{s+1},\dots,dx_N)$ and conclude by telescopic sum {over $s$}.%\comment{Referee 2: telescopic sum indexed by what? Done}
\end{proof}

We now proceed to the duality and attainment questions. First:

\begin{lemma}\label{lem:quantcausalcontinuous}
Suppose Assumption \ref{as:weakcont} holds. For $g\in C_b(\R^N)$, the functions
$$\textstyle (x_1,\dots,x_N)\mapsto g(x_1,\dots,x_{N}) - \int g(x_1,\dots,x_t,{x}_{t+1},\dots,{x}_N)\mu^{x_1,\dots,x_t}(d{x}_{t+1},\dots,d{x}_{N}),$$
belong themselves to $C_b(\R^N)$.

%We have $\gamma \in \Pi_c(\mu, \nu)$ if and only if $\gamma \in \Pi(\mu, \nu)$ and for $ t = 1, \dots, N$, every $\N\in C_b(\R^N) $ such that  $\int \N(x_1,\dots, x_t,\bar{x}_{t+1},\dots,\bar{x}_N) \mu^{x_1,\dots,x_t}(d \bar{x}_{t+1},\dots,d\bar{x}_N)=0$ and all $h \in C_b(\R^t)$, it holds: 
%$$\int  h(y_1,  \dots, y_t) \N(x_1,\dots, x_N)   d\gamma = 0 .$$ 
%Equivalently, $\gamma \in \Pi(\mu, \nu)$ and for all $ t \in\{ 1, \dots, N\}$, $h \in C_b(\R^t)$, $g\in C_b(\R^N)$ we have 
%%
%\begin{multline*}
%\int h(y_1,\dots,y_t)\Big\{ g(x_1,\dots,x_{N})- \Big. \\ \int\Big.  g(x_1,\dots,x_t,\bar{x}_{t+1},\dots,\bar{x}_N)\mu^{x_1,\dots,x_t}(d\bar{x}_{t+1},\dots,d\bar{x}_{N}) \Big\}d\gamma =0. 
%\end{multline*}
%%

\end{lemma}

\begin{proof}
It is suffices to check the continuity of 
$$\textstyle (x_1,\dots,x_t)\mapsto  \int g(x_1,\dots,x_t,{x}_{t+1},\dots,{x}_N)\mu^{x_1,\dots,x_t}(d{x}_{t+1},\dots,d{x}_{N}).$$
Let $x^n:=(x_1^n,\dots,x^n_t)$ converge to $y:=(y_1,\dots,y_t)$, thus we may assume all $x^n$ belongs to a fixed ball $B$ around $y^n$. Let us fix an arbitrary $\varepsilon >0$. By assumption $\mu^{x^n}$ converges weakly to $\mu^y$ and so in particular $\{\mu^{x^n}\}_{n\in\mathbb{N}}$ is tight. Thus we may find a compact in $\R^{N-t}$ such that $\sup_n\mu^{x^n}(K^c)\leq \varepsilon/3$. For large enough $n_0$ we also have that $\sup_{n>n_0,z\in K}|g(x^n,z)-g(y,z)|\leq \varepsilon/3$, since $g$ restricted to $B\times K$ must be uniformly continuous. We then write:
\begin{align*}\textstyle
\left|\int g(x^n,z)\mu^{x^n}(dz)-\int g(y,z)\mu^{y}(dz)\right|\leq \alpha + \beta + \gamma,
\end{align*}
with $$\textstyle \alpha = \left|\int\limits_K [g(x^n,z)- g(y,z)]\mu^{x^n}(dz)\right|, \beta = \left|\int\limits_{K^c} [g(x^n,z)- g(y,z)]\mu^{x^n}(dz)\right|, \gamma=\left|\int  g(y,z) [\mu^{x^n}(dz)- \mu^{y}(dz)]\right| .$$ We easily see that each term is smaller that $\varepsilon/3$ for $n$ large enough.
%For the second statement, and in light of item $2$ of Lemma \ref{lem:quantcausal}, it suffices to prove that the given $\N$ can be approximated in $L^2(\mu)$-norm by functions $ H^n\in C_b(\R^N)$ satisfying $$\int H^n(x_1,\dots, x_t,\bar{x}_{t+1},\dots,\bar{x}_N) \mu^{x_1,\dots,x_t}(d \bar{x}_{t+1},\dots,d\bar{x}_N)=0,$$  as well. For this, we first take $G^n\in C_b(\R^N)$ approximating $\N$ in $L^2(\mu)$-norm (we could even take them of compact support in this context) and define $$H^n(x_1,\dots,x_N):=G^n(x_1,\dots,x_N)  - \int G^n(x_1,\dots, x_t,\bar{x}_{t+1},\dots,\bar{x}_N) \mu^{x_1,\dots,x_t}(d \bar{x}_{t+1},\dots,d\bar{x}_N),$$
%which thanks to the first item of this proposition, satisfies the desired continuity property. To conclude we only need to show that the second term on the r.h.s.\ above converges to zero in $L^2(\mu)$. This follows easily by disintegration and Jensen:
%$$\int|G^n-\N|^2d\mu\geq\int\left\vert \int G^n(x_1,\dots,x_t,\bar{x}_{t+1},\dots,\bar{x}_N) \mu^{x_1,\dots,x_t}(d \bar{x}_{t+1},\dots,d\bar{x}_N) \right \vert^2\mu(dx_1,\dots,dx_t)$$
\end{proof}

\medskip

\begin{proof}[Proof of Theorem \ref{teo:duality}] By Proposition \ref{prop:equiv charact}, item $3$, we know that the set $\Pi_c(\mu,\nu)$ is the intersection of the compact $\Pi(\mu,\nu)$ with all the subspaces defined by the functions in \eqref{set of test functions}, each of them closed owing to Lemma \ref{lem:quantcausalcontinuous}.  Thus $\Pi_c(\mu,\nu)$ is compact and primal attainment follows easily. If the cost function belongs to $ C_b(\R^N\times\R^N)$, this theorem can be seen as a particular case of the Monge-Kantorovich problem with additional linear constraints considered in \cite[Theorem 2.5]{Zaev}, 
%Using the notation from that paper, denote
%$$W = \Bigg \{ h(Y_1) \Big[  g(X_1,X_2) - \EXP_\mu [g(X_1,X_2)\mid X_1] \Big], h \in \C_b, g \in L^{\infty}(\mu) \Bigg \} , $$
%$$\Pi_W = \Bigg \{  \gamma \in \Pi (\mu,\nu) \text{ s.t. }law \Big(Y_1 \mid X_1, X_2\Big) = law \Big(Y_1 \mid X_1\Big) \Bigg \} = \Pi_c(\mu,\nu)$$
where we just have to show that $\bbF \subset C_b(\R^N\times\R^N)$;
% where 
%\begin{eqnarray*}
%C_L(\mu, \nu) =& \Bigg \{ r \in C(\R^N\times \R^N) : \exists f = f_1 \oplus  f_2 \text{ s.t. } \\
%& |r| \leq f, f_1 \in L^1(\R^N,\mu)\cap C(\R^N), f_2 \in L^1(\R^N,\mu)\cap C(\R^N) \Bigg\},
%\end{eqnarray*}
%$ C_L(\mu) = \Big \{ f \in  L^1(E, \mu) \cap C(E) \Big\},$ $ C_L(\nu) = \Big \{ f \in  L^1(E, \nu) \cap C(S) \Big\}.$
%\end{proof}
%
%\begin{proposition}
% In the defined setting $W \subset C_L(\eta)$, where $\eta = (\mu, \nu)$.
% \end{proposition}
%\begin{proof}
but this is clear again by Lemma \ref{lem:quantcausalcontinuous}. On the other hand, when working with $\mathbb{S}\subset C_b(\R^N\times\R^N)$, it is easy to see that the same lemma implies that the martingales $M$ can be assumed continuous in the context of Proposition \ref{prop:equiv charact} and its proof, so again \cite[Theorem 2.5]{Zaev} establishes our stated result.
%
%$$
%\begin{array}{l}
% \Bigg| h_t(y_1,\dots,y_t) \Big[  g_t(x_1,\dots,x_{t+1}) -  \int g_t(x_1,\dots,x_t,\bar{x}_{t+1})\mu^{x_1,\dots,x_t}(d\bar{x}_{t+1})\Big]\Bigg|\\
%  \leq   \Big| h_t(y_1,\dots,y_t) \Big| \Big|  g_t(x_1,\dots,x_{t+1}) -  \int g_t(x_1,\dots,x_t,\bar{x}_{t+1})\mu^{x_1,\dots,x_t}(d\bar{x}_{t+1}) \Big| \\ 
% \leq  2 \| h_t \|_{\infty} \| g_t \|_{\infty}.
%\end{array}
%$$
%%
%By assumption $g \in L^{\infty}(\mu)$, hence, the last value is finite.
Finally duality for lower semicontinuous cost functions is achieved by the usual approximation arguments in \cite[Chapter 1]{Villani}, owing to the compactness of the set of all causal plans.
\end{proof}
\medskip
%WRITE CORRESPONDING RESULT IN THE BICAUSAL CASE

The analogue of Theorem \ref{teo:duality} is obtained for the bicausal setting. As in the causal case, we define the set
\begin{equation}\textstyle
\resizebox{.81\hsize}{!}{$
\bbF'  := 
\left\{
\begin{array}{c}
 F:\R^N\times\R^N\to\R \mbox{ s.t. } F(x_1,\dots,x_N,y_1,\dots,y_N) =  \\
\sum\limits_{t<N} h_t(y_1,\dots,y_t)\left [g_t(x_1,\dots,x_{N}) - \int g_t(x_1,\dots,x_t,{x}_{t+1},\dots, {x}_N)\mu^{x_1,\dots,x_t}(d{x}_{t+1},\dots,d{x}_N) \right ] +  \\
\sum\limits_{t<N} h'_t(x_1,\dots,x_t)\left [g'_t(y_1,\dots,y_{N}) - \int g'_t(y_1,\dots,y_t,{y}_{t+1},\dots, {y}_N)\nu^{y_1,\dots,y_t}(d{y}_{t+1},\dots,d{y}_N) \right ] ,  \\
\mbox{with}\,\, h_t, h'_t\in C_b(\R^t),\, \,g_t, g'_t\in C_b(\R^{N}) \,\,\mbox{for all }t<N  
\end{array}
\right\}. $}
\end{equation}

The following strengthened version of Assumption \ref{as:weakcont} will be needed:

\begin{assumption}\label{as:biweakcont}
Both $\mu$ and $\nu$ are successively weakly continuous.
\end{assumption}

\begin{corollary}
\label{coro:basic}
Suppose that $c:\R^N\times\R^N\to\R$ is lower semicontinuous and bounded from below, and that Assumption \ref{as:biweakcont} holds. Then there is no duality gap:
\begin{align*} \label{Dbc}\textstyle
\inf\limits_{\gamma \in \Pi_{bc}(\mu,\nu)} \int c d\gamma 
%&= \inf\limits_{\gamma \in \Pi(\mu,\nu)} \sup\limits_{h,g} \left[ \int c d\gamma + \int F  d\gamma %\right]\\ 
&=\sup\limits_{\substack{\Phi,\Psi\in C_b(\R^N),  F'\in\bbF' \\\Phi  \oplus \Psi  \leq c  +F'}} \textstyle \left[ \int \Phi d\mu + \int \Psi d\nu\right] . \tag{Dbc}
\end{align*}
Moreover, the infimum in the l.h.s.\ is attained. 
\end{corollary}

We omit the obvious formulation with discrete stochastic integrals, as well as the proof.

%\begin{proof}
%Same as in Theorem \ref{teo:duality}, after employing Proposition \ref{prop:equiv charact} in both directions.
%\end{proof}
 
\begin{remark}\label{rem:first_Monge}
%For a measurable $T:E\to S$, it trivially holds that $\gamma^T\in \Pi(\mu,T_*\mu)$. It is clear that 
%%
%$$\theta^x_{\gamma^T}(dy)=\delta_{T(x)}(dy),$$
%%
%and so $\gamma^T\in \Pi_c(\mu,T_*\mu)$ if and only if $T$ is measurable as a map from $(E,\F_t(E))$ into $(S,\F_t(S))$, for each $t$. More generally, of course, any $\mu$-a.s.\ version $\tilde{T}$ of such a  $T$ fulfils $\gamma^{\tilde{T}}\in \Pi_c(\mu,T_*\mu)$. Bicausality of $\gamma^T$ is more subtle. If for instance $T$ admitted a measurable $\mu$-a.s.\ left inverse\footnote{In the present setting, a measurable $T_*\mu$-a.s.\ right-inverse of $T$ always exist, by a measurable selection of the (Borel) graph of T}, i.e.\ a measurable $L:S\to E$ s.t.\ $\mu$-a.s.\ $L\circ T=id$ holds, then bicausality requires that $L:(S,\F_t(S))\to (E,\F_t(E))$ be measurable for each $t$, $T_*\mu$-a.s. However such a left-inverse need not exist. We comment again on this in the simpler discrete case in Remark \ref{rem:second_Monge}, where we exhibit a class of bicausal maps.
% On the other hand, if $\gamma^T$ is causal and 
We have observed that a measurable $T:\R^N\to \R^N$ is causal if and only if it is adapted, in the sense that $\mu$-a.s.\ $T(x_1,\dots,x_N)\,\,=\,\, (T^1(x_1),T^2(x_1,x_2),\dots,T^N(x_1,\dots,x_N))$ for measurable $T^i:\R^t\to\R$. Bicausality is more subtle, but clearly holds if for instance $T$ admits a measurable $\mu$-a.s.\ left inverse which is also adapted.
\end{remark}

\section{Dynamic programming principle}\label{DPP causal case}

Let $\gamma\in\Pi(\R^N,\R^N)$. It is then possible to decompose $\gamma$, uniquely in a suitable way, by successive disintegration first w.r.t.\ $(x_1,\dots,x_{N-1},y_1,\dots,y_{N-1})$, then w.r.t.\ $(x_1,\dots,x_{N-2},$ $y_1,\dots,y_{N-2})$ and so forth in a recursive way, obtaining
\begin{equation}\textstyle
\resizebox{.81\hsize}{!}{$
\gamma(dx_1,\dots,dx_N,dy_1,\dots,dy_N)=\bar{\gamma}(dx_1,dy_1)\gamma^{x_1,y_1}(dx_2,dy_2)\dots \gamma^{x_1,\dots,x_{N-1},y_1,\dots,y_{N-1}}(dx_N,dy_N), $}
\end{equation}
so that each $\gamma^{x_1,\dots,x_{t},y_1,\dots,y_{t}}(dx_{t+1},dy_{t+1})$ is a regular conditional probability or kernel (see \cite[Theorem 10.4.14, Corollary 10.4.17]{Bogachev} and \cite[Lemma 1.41 ]{Kallbook}). We will freely perform concatenation of these objects or further decompose them into smaller kernels, as justified in \cite[Chapter 7.4.3]{BertsekasShreve}. The characterization of causal plans through the kernels was given in {Point 2 of} Proposition \ref{prop:equiv charact}. %%\comment{Referee 2: say ``Point 2''. Done}

We believe that generally \eqref{Pc} does not allow for a meaningful and useful ``factorized'' or ``recursive'' formulation, along the lines of what a dynamic programming principle (DPP) ought to be. However, we show first that the causal quasi-Markov transport plans we will introduce, and their associated optimal transport problem, do possess a DPP. This fact (Theorem \ref{thm:markovDPP}) together with Proposition \ref{prop:markov=cuasal} will then allow us to prove Theorem \ref{thm:DPPcausal}.

\begin{definition}
A causal transport plan $\gamma\in\Pi_c(\mu,\nu)$ is called \textit{causal quasi-Markov}, which we denote by $\gamma\in\Pi_{cqm}(\mu,\nu)$, if $\gamma^{x_1,\dots,x_t,y_1,\dots,y_t}(dx_{t+1},dy_{t+1})=\gamma^{x_t,y_1,\dots,y_t}(dx_{t+1},dy_{t+1})$ for every $t$.
\end{definition}

%\comment{Referee 2: who belongs to $\Pi_{cqm}(\mu,\nu)$?. Done}
Of course $\Pi_{cqm}(\mu,\nu)\neq\emptyset $ iff $ \mu$ is a Markov measure. {To wit, if $\gamma\in \Pi_{cqm}(\mu,\nu)$ then $\gamma^{x_t,y_1,\dots,y_t}(dx_{t+1})=\mu^{x_1,\dots,x_t}(dx_{t+1})$, which implies the Markov property of $\mu$. Conversely, if this property holds, one can check that the independent coupling of $\mu$ and $\nu$ belongs to $\Pi_{cqm}(\mu,\nu)$. } In either case,  $\textstyle \gamma^{x_1,\dots,x_t,y_1,\dots,y_t}(dx_{t+1})=\mu^{x_t}(dx_{t+1})$ must be satisfied. 
%The following lemma is consequence of a trivial adaptation of Proposition \ref{prop:equiv charact} ($(1)=(2)$), so we omit it.
%
%\comment{Perhaps we do not need this Lemma}
%\begin{lemma}\label{lem:markov}
%Assume $\mu$ is markovian. Then the following are equivalent:
%\begin{enumerate}
%\item  $\gamma\in\Pi_{cqm}(\mu,\nu)$.
%\item If $\gamma$ is decomposed in terms of regular kernels as follows,
%%
%\begin{equation}\label{successivedisintegration_markov}
%\resizebox{.81\hsize}{!}{$
%\gamma(dx_1,\dots,dx_N,dy_1,\dots,dy_N)=\bar{\gamma}(dx_1,dy_1)\gamma^{x_1,y_1}(dx_2,dy_2)\dots \gamma^{x_1,\dots,x_{N-1},y_1,\dots,y_{N-1}}(dx_N,dy_N), $}
%\end{equation}
%%
%then $\bar{\gamma}\in \Pi(p^1_*\mu,p^1_*\nu)$, for $t<N$ and $\gamma$-almost every $x_1,\dots,x_{t},y_1,\dots,y_{t}$  $$p^1_*\gamma^{x_1,\dots,x_{t},y_1,\dots,y_{t}}\,\,\, \,\,=\,\,\,\,\, \mu^{x_{t}},$$
%and for $\nu$-almost every $y_1,\dots,y_{t}$ the following equality (as measures) holds: 
%$$\int\limits_{x_{t}} \gamma^{x_{t},y_1,\dots,y_{t}}(\R,dy_{t+1})\gamma^{y_1,\dots,y_{t}}(dx_{t}) \,\,\, \,\,=\,\,\, \,\,\nu^{y_1,\dots,y_{t}}(dy_{t+1}).$$
%\end{enumerate}
%\end{lemma}
We introduce the causal quasi-Markov transport problem:
\begin{equation}
\label{Pcqm}\tag{Pcqm}\textstyle
\inf\limits_{\gamma\in \Pi_{cqm}(\mu,\nu)}\int cd\gamma
\end{equation}
We now prove the DPP for \eqref{Pcqm}:
\begin{theorem}[\bf{DPP for \eqref{Pcqm}}]\label{thm:markovDPP}
Let $\mu, \nu \in \mathcal{P}(\R^N)$, suppose that $\mu$ is Markov and that the cost is semiseparable in the sense that $c=\sum\limits_t c_t(x_t,y_1,\dots,y_t)$ for non-negative Borel functions $c_t$. Then starting from $V^c_N:=0$ and defining recursively for $t=N,\dots, 2$:
%
%\comment{Referee 2: why are there two ``$dx_t$''? Done}
\begin{multline}\textstyle
V^c_{t-1}(y_1,\dots,y_{t-1};m(dx_{t-1}))\,\,:= \\ \textstyle \inf\limits_{\gamma\in \Pi\left(\int_{x_{t-1}}m(dx_{t-1})\mu^{x_{t-1}}(dx_t)\, ,\, \nu^{y_1,\dots,y_{t-1}}(dy_t)\right)} \int \gamma(dx_t,dy_t)\bigl\{c_t(x_t,y_1,\dots,y_t)  \\
  + V^c_{t}(y_1,\dots,y_t;\gamma^{y_t}_t)\bigr\},\label{weak-trans-cqm}
\end{multline}
we have that 
$$\textstyle V^c_0:= \inf_{\gamma\in \Pi(p^1_*\mu,p^1_*\nu )} \int \gamma(dx_1,dy_1) \{c_1(x_1,y_1) + V^c_1(y_1;\gamma^{y_1}_1)\}\, =\, {value\eqref{Pcqm}}.$$
% 
%i.e.\ the recursion determines \eqref{Pc}. Furthermore, every optimal causal transport has the additional property that for all $t$ holds $\gamma^{x_1,\dots,x_t,y_1,\dots,y_t}(dx_{t+1},dy_{t+1})= \gamma^{x_t,y_1,\dots,y_t}(dx_{t+1},dy_{t+1}) $, and 
%\comment{Referee 2: is this a result or assumption?. Done?}
{Furthermore,} each function $V^c_{t-1}$ is jointly universally measurable and convex in its last component. If moreover Assumption \ref{as:weakcont} holds and each $c_t$ is l.s.c, then $V^c_{t-1}(y_1,\dots,y_{t-1};\cdot)$ is l.s.c.\ and Problem \eqref{Pcqm} is attained.
\end{theorem} 

To be precise, we shall prove that $V^c_{t-1}$ is lower semianalytic (i.e.\ $\{V^c_{t-1}<r\}$ is analytic for each $r$; see \cite[Definition 7.21]{BertsekasShreve}), which implies universal measurability.\\

\begin{proof}
%\comment{Referee 2:explain convexity. Done}
We split the proof in several steps. It is clear from its definition that $V^c_{t}$ (for any $t$) is convex in its last component {(namely $V_t^c(y_1,\dots,y_t;\lambda m+(1-\lambda)\bar m)\leq \lambda V^c_t(y_1,\dots,y_t; m)+(1-\lambda)V^c_t(y_1,\dots,y_t; \bar{m})$, see e.g.\ \cite[Theorem 4.8]{Villani_Old})} and generally bounded from below.\\

\textit{STEP 1:} We first show that the sets:
$$\textstyle
D_{t-1}:=\bigl \{ (y_1,\dots,y_{t-1},m,\gamma)\mbox{ s.t. }\gamma\in \Pi\bigl(\,\,\int_{x_{t-1}}m(dx_{t-1})\mu^{x_{t-1}}(dx_t)\, ,\, \nu^{y_1,\dots,y_{t-1}}(dy_t)\bigr)\bigr\},
$$
are Borel, so in particular analytic. Indeed, observe first that 
$$\textstyle
\bar{D}_{t-1}:=\left \{ (p,q,\gamma)\mbox{ s.t. }\gamma\in \Pi\left(q(dx_t)\, ,\, p(dy_t)\right)\right\},
$$
is closed w.r.t.\ weak convergence of probability measures. So if we denote by $\Psi$ the map $$\textstyle (y_1,\dots,y_{t-1},m,\gamma)\mapsto  \bigl(\nu^{y_1,\dots,y_{t-1}}(dy_t),\int_{x_{t-1}}m(dx_{t-1})\mu^{x_{t-1}}(dx_t),\gamma \bigr),$$ we have $D_{t-1}=\Psi^{-1}(\bar{D}_{t-1})$ and so it suffices to show that the map $m(dx_{t-1})\mapsto \int\limits_{x_{t-1}}m(dx_{t-1})\mu^{x_{t-1}}(dx_t)$ is Borel. %\comment{Referee 2: show how things boil down to simple statement \eqref{borelbounded}. Done?} 
{For this, it suffices to show that for Borel sets $\{A_i,B_i\}_{i=1}^r$ the sets of the form $$\textstyle \{m:\, \int_{x_{t-1}} m(dx_{t-1})\mu^{x_{t-1}}(A_i)\in B_i\mbox{ for }i=1,\dots,r \},$$ are again Borel. Since $x_{t-1}\mapsto \tilde{g}_i(x_{t-1}):=\mu^{x_{t-1}}(A_i)$ is a bounded Borel function, everything ultimately}  boils down to proving that if $g_1,\dots,g_r$ are bounded Borel functions, then the sets
\begin{equation}\textstyle
\left\{m\mbox{ Borel prob.\ measures s.t. }\int g_1dm\geq 0,\dots,\int g_rdm\geq 0\right\},\label{borelbounded}
\end{equation}
are Borel. But this is now a straightforward monotone class argument which we omit. %The general case follows from an application of the functional monotone class theorem.

For future use in Step 5, we observe that the sets 
$$\textstyle\bigl \{ (m,\gamma)\mbox{ s.t. }\gamma\in \Pi\bigl(\,\,\int_{x_{t-1}}m(dx_{t-1})\mu^{x_{t-1}}(dx_t)\, ,\, \nu^{y_1,\dots,y_{t-1}}(dy_t)\bigr)\bigr\},$$
which are generally Borel only (as a fibers of the set $D_{t-1}$), are further closed as soon as Assumption \ref{as:weakcont} holds. \\

%As a matter of fact, if $(y^n_1,\dots,y^n_{t-1},m^n,\gamma^n)\in D_{t-1}$ and 
%%
%$$
%(y^n_1,\dots,y^n_{t-1},m^n,\gamma^n)\to (y_1,\dots,y_{t-1},m,\gamma),
%$$
%%
%where for $m^n$ and $\gamma^n$ we use weak convergence of probability measures, then it is easy to see out of our Assumption \eqref{as:biweakcont} that $\int\limits_{x_{t-1}}m^n(dx_{t-1})\mu^{x_{t-1}}(dx_t)$ converges weakly to $\int\limits_{x_{t-1}}m(dx_{t-1})\mu^{x_{t-1}}(dx_t)$ and of course $\nu^{y_1^n,\dots,y^n_{t-1}}(dy_t)$ converges weakly to $ \nu^{y_1,\dots,y_{t-1}}(dy_t)$, which implies that $\gamma$ belongs to  $\Pi\left(\int\limits_{x_{t-1}}m(dx_{t-1})\mu^{x_{t-1}}(dx_t)\, ,\, \nu^{y_1,\dots,y_{t-1}}(dy_t)\right)$, so $D_{t-1}$ is actually closed.
%\\

\textit{STEP 2:} Let us now prove that the recursion is well-defined; namely, that the involved integrated functions are defined at most except for a null-set w.r.t.\ the integral. To be precise, we will show that these functions are lower semianalytic, and so universally measurable, thus in particular their integrals are well defined w.r.t.\ any Borel measure. For that matter, we shall first prove that 
$$\textstyle y_1,\dots,y_{t},m(dx_{t})\mapsto V^c_{t}(y_1,\dots,y_{t};m) $$ 
is lower semianalytic (l.s.a.\ in short). By \cite[Lemma 1.40]{Kallbook}, we know that for the regular kernel of a given $\gamma(dx_{t},dy_{t})$, the application $y_{t}\mapsto \gamma^{y_{t}}(dx_{t})$ is Borel measurable, then so is
the map $(y_1, \dots, y_{t},\gamma) \mapsto (y_1, \dots, y_{t}, \gamma^{y_{t}}(dx_{t}))$. Therefore, by \cite[Lemma 7.30(3)]{BertsekasShreve} (on the composition of l.s.a.\ and Borel maps)  we deduce that $(y_1,\dots,y_{t},\gamma)\mapsto V^c_{t}(y_1,\dots,y_{t};\gamma^{y_{t}}) $ is l.s.a\ for $\gamma$ participating in the infimum in \eqref{weak-trans-cqm} and in particular the integral is indeed well-defined.

We start with $V^c_{N-1}$. The cost 
$$\textstyle (y_1,\dots,y_{N-1},m,\gamma)\mapsto \int \gamma(dx_N,dy_N)c_N(x_N,y_1,\dots,y_N),$$
is Borel measurable and so in particular l.s.a. To see this, take $g_1=-[c_N\wedge s]$ in \eqref{borelbounded} and take $s\to +\infty$. We can write $V^c_{N-1}(y_1,\dots,y_{N-1};m)$ as the infimum of this cost over the fiber of the set $D_{N-1}$ at $(y_1,\dots,y_{N-1},m)$. By \cite[Proposition 7.47]{BertsekasShreve} and Step $1$, the function $V^c_{N-1}$ is l.s.a.\ on the projection of $D_{N-1}$ onto its $(y_1,\dots,y_{N-1},m)$ components. By reverse induction, let us suppose that $V^c_{t}$ is l.s.a.\ and prove that $V^c_{t-1}$ is likewise. The cost to consider is now
$$\textstyle (y_1,\dots,y_{t-1},m,\gamma)\mapsto \int \gamma(dx_t,dy_t)c_t(x_t,y_1,\dots,y_t)+\int\nu^{y_1,\dots,y_{t-1}}(dy_t)V^c_{t}(y_1,\dots,y_t;\gamma^{y_t}_t),$$
the first term of which is Borel as before and whose second term is l.s.a.\ by virtue of the inductive step, the discussion about composition of l.s.a.\ and Borel maps above, and by \cite[Proposition 7.48]{BertsekasShreve} on the integration of l.s.a.\ functions with respect to Borel kernels. By \cite[Lemma 7.30(4)]{BertsekasShreve} we get that their sum is l.s.a.\ and so by \cite[Proposition 7.47]{BertsekasShreve} and Step $1$ again we conclude that $V^c_{t-1}$ is l.s.a.\\

\textit{STEP 3:} By the previous point, and the way we wrote \eqref{weak-trans-cqm} as an infimum of a l.s.a.\ cost over a fiber of an analytic set in Step $2$, we can perform for any $\varepsilon>0$ by \cite[Proposition 7.50(b)]{BertsekasShreve} a selection 
$$\textstyle
(y_1,\dots,y_{t-1},m)\mapsto L_{t-1,\varepsilon}^{y_1,\dots,y_{t-1};m}(dx_t,dy_t)\in \Pi\left(\,\,\int_{x_{t-1}}m(dx_{t-1})\mu^{x_{t-1}}(dx_t)\, ,\, \nu^{y_1,\dots,y_{t-1}}(dy_t)\right),
$$
so that the above mapping is universally measurable (i.e.\ measurable w.r.t.\ any completion of the corresponding Borel sets of its domain) and 
\begin{multline*}\textstyle
V^c_{t-1}(y_1,\dots,y_{t-1},m)+\varepsilon\geq \int L_{t-1,\varepsilon}^{y_1,\dots,y_{t-1};m}(dx_t,dy_t)\bigl[c_t(x_t,y_1,\dots,y_t) \bigr . \\ \bigl .+V^c_{t}(y_1,\dots,y_t;(L_{t-1,\varepsilon}^{y_1,\dots,y_{t-1};m})^{y_t})\bigr ],
\end{multline*}
so each $L$ is an $\varepsilon$-optimizer for the corresponding problem. We will now build a measure, whose successive kernels will solve the recursions \eqref{weak-trans-cqm} at each step, modulo an $\varepsilon$ margin. Start with any $\varepsilon$-optimizer $\gamma_{0,\varepsilon}(dx_1,dy_1)$ of $V^c_0$. Then take $y_{1}\mapsto\gamma_{1,\varepsilon}^{y_1}(dx_2,dy_2):=L_{1,\varepsilon}^{y_1;\gamma_{0,\varepsilon}^{y_1}}(dx_2,dy_2)$, which is universally measurable by the composition result \cite[	Proposition 7.44]{BertsekasShreve}. Inductively, if $(y_1,\dots,y_{t-1})\mapsto\gamma_{t-1,\varepsilon}^{y_1,\dots,y_{t-1}}(dx_t,dy_t)$ has been constructed in a universally measurable way, we build 
$$\textstyle (y_1,\dots,y_t)\mapsto \gamma_{t,\varepsilon}^{y_1,\dots,y_t}(dx_{t+1},dy_{t+1}):= L_{t,\varepsilon}^{y_1,\dots,y_t;(\gamma_{t-1,\varepsilon}^{y_1,\dots,y_{t-1}})^{y_t}}(dx_{t+1},dy_{t+1}),$$
which is universally measurable by \cite[Proposition 7.44]{BertsekasShreve}, the inductive step, and the fact that the function that associates a regular kernel to a product measure is Borel. This finishes the induction, and by construction 
\begin{align}\label{epsilon opt kernel}\textstyle
\gamma_{t,\varepsilon}^{y_1,\dots,y_t}(dx_{t+1},dy_{t+1}) \in \Pi\left( \,\,\int_{x_t}(\gamma_{t-1,\varepsilon}^{y_1,\dots,y_{t-1}})^{y_t}(dx_t)\mu^{x_t}(dx_{t+1}) ,\nu^{y_1,\dots,y_t}(dy_{t+1})\right ),
\end{align}
and $\gamma_{t,\varepsilon}^{y_1,\dots,y_t}$ attains $V^c_t(y_1,\dots,y_t;(\gamma_{t-1,\varepsilon}^{y_1,\dots,y_{t-1}})^{y_t})$ except for an $\varepsilon$ margin. We now define
$$\textstyle (x_t,y_1,\dots,y_t)\mapsto\Gamma_{t,\varepsilon}^{x_t,y_1,\dots,y_t}(dx_{t+1},dy_{t+1}):=\mu^{x_t}(dx_{t+1})(\gamma_{t,\varepsilon}^{y_1,\dots,y_t})^{x_{t+1}}(dy_{t+1}),$$
which is universally measurable by  \cite[	Proposition 7.44]{BertsekasShreve} again. \mbox{By \cite[Proposition 7.45]{BertsekasShreve} } the successive concatenation of these kernel induce a unique Borel measure $\Gamma_{\varepsilon}$ such that
\begin{multline*}\textstyle
\Gamma_{\varepsilon}(dx_1,\dots,dx_N,dy_1,\dots,dy_N)\,=\, \gamma_{0,\varepsilon}(dx_1,dy_1)\Gamma_{1,\varepsilon}^{x_1,y_1}(dx_2,dy_2)\dots\\\Gamma_{t,\varepsilon}^{x_t,y_1,\dots,y_t}(dx_{t+1},dy_{t+1})\dots  \Gamma_{N-1,\varepsilon}^{x_{N-1},y_1,\dots,y_{N-1}}(dx_{N},dy_{N}).
\end{multline*}
By construction, this measure is causal quasi-Markov and its $x$-marginal is exactly $\mu$. As for the $y$-marginal, it is easy to see that $\Gamma_{\varepsilon}(dy_1)=\nu(y_1)$ and that 
\begin{multline*}\textstyle \Gamma_{\varepsilon}(dy_1,dy_2)=\int_{x_1,x_2}\gamma_{0,\varepsilon}(dx_1,dy_1)\mu^{x_1}(dx_2)(\gamma_{1,\varepsilon}^{y_1})^{x_2}(dy_2)\\ \textstyle =\nu(dy_1)\int_{x_2}\left(\,\,\int_{x_1}\gamma_{0,\varepsilon}^{y_1}(dx_1)\mu^{x_1}(dx_2)\right)(\gamma_{1,\varepsilon}^{y_1})^{x_2}(dy_2)=\nu(dy_1)\nu^{y_1}(dy_2),
\end{multline*}
observing (from (\ref{epsilon opt kernel})) in the last line, inside the brackets, that this is exactly the first marginal of $\gamma_{1,\varepsilon}^{y_1}$. Inductively, one can verify that $\Gamma_{\varepsilon}$ has $\nu$ as its $y$-marginal.\\

\textit{STEP 4:} We now prove the equality $V^c_0 = value\eqref{Pcqm}$. By the previous step, $V^c_0+N\varepsilon \geq value\eqref{Pcqm}$, since $\Gamma_{\varepsilon}$ is feasible for \eqref{Pcqm} and because by construction $\Gamma_{\varepsilon}$ was designed ${\varepsilon}$-optimal at each step, so that overall $V^c_0+N{\varepsilon}\geq\int cd\Gamma_{\varepsilon}$. By letting $\varepsilon\to 0$, we get $V^c_0 \geq value\eqref{Pcqm}$. On the other hand, given any $\gamma\in \Pi_{cqm}(\mu,\nu)$, we clearly have that
%let us call $m^{y_1,\dots,y_t}(dx_{t+1},dy_{t+1}):= \gamma^{y_1,\dots,y_t}(dx_{t+1},dy_{t+1})$. Then $m^{y_1,\dots,y_t}(dy_{t+1}) = \nu^{y_1,\dots,y_t}(dy_{t+1}) $ and further 
%
%$$m^{y_1,\dots,y_t}(dx_{t+1})=\int_{x_t}\gamma^{y_1,\dots,y_t}(dx_t)\gamma^{x_t,y_1,\dots,y_t}(dx_{t+1})=\int_{x_t}(m^{y_1,\dots,y_{t-1}})^{y_t}(dx_t)\mu^{x_t}(dx_{t+1}).$$
%
%This proves that
\begin{equation}\label{eq:aux}\textstyle
\gamma^{y_1,\dots,y_{t}}(dx_{t+1},dy_{t+1}) \in \Pi\bigl( \,\, \int_{x_{t}} (\gamma^{y_1,\dots,y_{t-1}})^{y_t} (dx_{t})\mu^{x_{t}}(dx_{t+1})\, ,\, \nu^{y_1,\dots,y_{t}}(dy_{t+1})   \bigr) ,
\end{equation}
 and further we can write 
\begin{multline}\textstyle
\int c_td\gamma =\int \gamma(dy_1,\dots,dy_{t-1}) \gamma^{y_1,\dots,y_{t-1}}(dx_t,dy_t)c_t(x_t,y_1,\dots,y_t)  = \\ \textstyle
\int\gamma(dx_1,dy_1)\gamma^{y_1}(dx_2,dy_2)\gamma^{y_1,y_2}(dx_3,dy_3)\dots\gamma^{y_1,\dots,y_{t-1}}(dx_t,dy_t)c_t(x_t,y_1,\dots,y_t),\notag
\end{multline}
so
\begin{equation}\textstyle
\int cd\gamma=\int \gamma(dx_1,dy_1)\left\{c_1 + \int\gamma^{y_1}(dx_2,dy_2)\left \{ c_2 + \int\gamma^{y_1,y_2}(dx_3,dy_3) \big\{ c_3+\dots \big \} \right \}  \right \}.\label{eq:crucial}
\end{equation}
Combining this with \eqref{eq:aux} we get $value\eqref{Pcqm} \geq V_0^c $ and conclude the desired equality.\\ %By Step $5$ we conclude that $\Gamma$ is optimal for \eqref{Pcqm}, which finishes the proof.\\
%
%, so in particular  $V^c_{N-1}(y_1,\dots,y_{N-1})\leq \int c_N\gamma^{y_1,\dots,y_{N-1}}(dx_N,dy_N)=: v_{N-1}(y_1,\dots,y_{N-1}) $, and then 

%{\color{red}THE REVERSE INEQUALITY SEEMS TO NEED A MEASURABLE SELECTION RESULT FOR GENERALIZED TRANSPORT PROBLEMS A LA GOZLAN }

From now on we take Assumption \ref{as:weakcont} for granted and assume l.s.c.\ costs.\\

\textit{STEP 5:} We establish now that $V^c_{t-1}$ is lower semicontinuous in its last argument (i.e.\ $m$), the other ones being fixed. The desired lower semicontinuity could be proved by hand using convex combination of kernels as in Step $6$ below, but a simpler argument is to invoke a ``maximum theorem'' as in \cite[Proposition 7.33]{BertsekasShreve}, which establishes the lower semicontinuity of a value function as long as the cost is jointly lower semicontinuous (in our case w.r.t.\ $(m,\gamma)$, since the $y$'s are fixed, and this clearly holds) and as soon as the constraint set is a fiber of a closed set (which is true by the last observation in Step $1$). The lower semicontinuity in $(m,\gamma)$ is proved by reverse induction, much as in Step $2$, and for \cite[Proposition 7.33]{BertsekasShreve} one can assume that the $\gamma$'s live a priori in a compact space, since varying the $m$'s in a tight set only will ``move'' the constraint set $\Pi$ inside a larger tight set in the product.\\

\textit{STEP 6:} We now prove that each of the problems in \eqref{weak-trans-cqm} is attained. The $\int c_td\gamma$ part of the cost is l.s.c.\ and linear, so it causes no trouble and we may assume $c_t\equiv 0$. We take a minimizing sequence $\gamma_n$ for \eqref{weak-trans-cqm}:
$$\textstyle V^c_{t-1}(y_1,\dots,y_{t-1};m)=\lim_n \int \nu^{y_1,\dots,y_{t-1}}(dy_t)V^c(y_1,\dots,y_{t};\gamma_n^{y_t}).$$
Trivially the sequence $\int_{y_t}\nu^{y_1,\dots,y_{t-1}}(dy_t)\gamma_n^{y_t}(dx_t)$ is tight, since it is identically equal to the measure $\int_{x_{t-1}}m(dx_{t-1})\mu^{x_{t-1}}(dx_t)$. By \cite[Theorem 3.15]{Baldernotes} we conclude that there is a subsequence of the kernel $\{y_t\mapsto \gamma_n^{y_t}\}_n$, which we denote the same, and a kernel $\{y_t\mapsto\gamma^{y_t}_t\}$ such that the sequence of C\'esaro averages 
$$\textstyle y_t\mapsto\tilde{\gamma}_n^{y_t}:=\frac{1}{n}\sum\limits_{i\leq n}\gamma_i^{y_t}$$
converges weakly, for $\nu^{y_1,\dots,y_{t-1}}(dy_t)$-a.e.\ $y_t$, to the kernel $y_t\mapsto\gamma^{y_t}$ (i.e.\ as measures on $x_t$, see \cite[Definition 3.10]{Baldernotes}) and further it holds that $\int\nu^{y_1,\dots,y_{t-1}}(dy_t)\gamma^{y_t}_t(dx_t)$ equals $\int_{x_{t-1}}m(dx_{t-1})\mu^{x_{t-1}}(dx_t)$ by \cite[Corollary 3.14]{Baldernotes}, so the measure $\gamma:=\nu^{y_1,\dots,y_{t-1}}(dy_t)\gamma^{y_t}_t(dx_t)$ is feasible for \eqref{weak-trans-cqm}. All in all it follows
\begin{align*}\textstyle
V^c_{t-1}(y_1,\dots,y_{t-1};m) &= \textstyle \liminf_n\frac{1}{n}\sum\limits_{i\leq n} \int \nu^{y_1,\dots,y_{t-1}}(dy_t)V^c(y_1,\dots,y_{t};\gamma_i^{y_t})\\  & \textstyle \geq 
 \liminf_n \int \nu^{y_1,\dots,y_{t-1}}(dy_t)V^c(y_1,\dots,y_{t};\tilde{\gamma}_n^{y_t})\\
&\textstyle \geq \int \nu^{y_1,\dots,y_{t-1}}(dy_t)\liminf_n V^c(y_1,\dots,y_{t};\tilde{\gamma}_n^{y_t})\\ \textstyle&= \textstyle\int \nu^{y_1,\dots,y_{t-1}}(dy_t) V^c(y_1,\dots,y_{t};{\gamma}^{y_t}_t),
\end{align*}
by convexity, Fatou's Lemma (remember the $V^c$'s are bounded below) and the lower semicontinuity established in the previous Step. Therefore, $\gamma$ is feasible and optimal.\\

%\comment{Referee 2: ``STEP 7''. Done}
\textit{STEP 7:} By the previous step, we may now go back to Step 3 and find by \linebreak \cite[Proposition 7.50(b)]{BertsekasShreve}  a selection of optimizers at each time (as opposed to only $\varepsilon$-optimizers) $L_t^{y_1,\dots,y_t;m}$. Then we can redo Step 3 with $\varepsilon=0$, building a global measure $\Gamma$ which exactly solves the recursion and is feasible for \eqref{Pcqm}, and so is optimal for it.
\end{proof}
\medskip

Some observations regarding the previous theorem and its proof:

\begin{remark}
Even if $c$ does not have the stated semiseparable structure, a look at the proof shows that the recursion is well-posed and that its value gives an upper bound for $value\eqref{Pcqm}$. The semiseparable structure is crucial for the lower bound (see \eqref{eq:crucial}) only.
\end{remark}

\begin{remark}
As a by-product of the previous proof, we see from Step 6 therein a way to prove attainability of general transport problems as in \cite{weaktransport} in the presence of lower semicontinuity and convexity of the cost w.r.t.\ the kernel, without the assumption that the state space be compact or discrete-like.
\end{remark}

As we have mentioned in Theorem \ref{thm:DPPcausal}, each optimization problem in DPP is of the form  of general (non-linear) transports on the line. Under Assumption \ref{as:biweakcont} and for the lower semicontinuous cost function each of the value functions in (\ref{weak-trans-cqm}) is convex in the kernel and jointly l.s.c. Assuming additionally that all the regular kernels of measures $\mu, \nu$ are compactly supported, we can use the duality result of Gozlan et al.\   \cite[Theorem 9.5]{weaktransport}. Denote $\eta(dx_t) = \int\limits_{x_{t-1}}m(dx_{t-1})\mu^{x_{t-1}}(dx_t)$, and 
$$\textstyle \bar{c}_t^{y_1,\dots,y_{t-1}}(y_t; p) = \int_{x_t} c_t(x_t,y_1,\dots,y_t)p(dx_t) + V^c_{t}(y_1,\dots,y_{t-1},y_t;p(dx_{t})).$$ 
Then the following dual formulation holds:
\begin{align}\label{dual DPP}\textstyle
V_{t-1}^c(y_1,\dots,y_{t-1}; m(dx_{t-1})) = \sup_{\phi} \bigl\lbrace \int R_{\bar{c}_t^{y_1,\dots,y_{t-1}}} \phi(y_t) \nu^{y_1,\dots,y_{t-1}}(dy_t) - \int \phi(x_t) \eta(dx_t) \bigr\rbrace,
\end{align}
where
$$\textstyle R_{\bar{c}_t^{y_1,\dots,y_{t-1}}} \phi(y_t) = \inf_{p \in \mathcal{P}(\R)} \left\lbrace  \int \phi(x_t)p(dx_t) + \bar{c}_t^{y_1,\dots,y_{t-1}}(y_t; p)\right\rbrace. $$

\medskip

We finally establish a sufficient condition for \eqref{Pcqm} and \eqref{Pc} to be equivalent:

\begin{proposition}\label{prop:markov=cuasal}
Let $\mu$ be Markov and the cost be semiseparable. Then $value\eqref{Pcqm}= value\eqref{Pc}$ and if \eqref{Pc} is attained, then there is some optimal causal transport which is further causal quasi-Markov.
\end{proposition} 

\begin{proof}
We prove that under the given assumptions, any causal plan has a related causal quasi-Markov plan which incurs in the same cost. Let $\gamma \in \Pi_c(\mu,\nu)$, and build
$$\textstyle d\hat{\gamma}=\gamma(dx_1,dy_1)\gamma^{x_1,y_1}(dx_2,dy_2)\gamma^{x_2,y_1,y_2}(dx_3,dy_3)\dots\gamma^{x_{N-1},y_1,\dots,y_{N-1}}(dx_N,dy_N). $$
%
%\comment{Referee 2: replace causal by quasi-Markov. Done}
As $\hat{\gamma}^{x_1,\dots,x_t,y_1,\dots,y_t}={\gamma}^{x_t,y_1,\dots,y_t}= \hat{\gamma}^{x_t,y_1,\dots,y_t} $ we have $\hat{\gamma}$ is {quasi-Markov} and its $x$-marginal is $\mu$, indeed
\begin{align*}\textstyle
\hat{\gamma}^{x_1,\dots,x_t,y_1,\dots,y_t}(dx_{t+1})=&\,\,{\gamma}^{x_t,y_1,\dots,y_t}(dx_{t+1})\\=&\textstyle\,\,\int_{x_1,\dots,x_{t-1}}\gamma^{x_1,\dots,x_t,y_1,\dots,y_t}(dx_{t+1}){\gamma}^{x_t,y_1,\dots,y_t}(dx_1,\dots,dx_{t-1})
\\ =&\textstyle\,\, \int_{x_1,\dots,x_{t-1}}\mu^{x_t}(dx_{t+1}){\gamma}^{x_t,y_1,\dots,y_t}(dx_1,\dots,dx_{t-1}) =\mu^{x_t}(dx_{t+1}),
\end{align*}
{so it is also causal.} We finally prove that for every function of the form $H:=H(x_t,y_1,\dots,y_t)$, holds that $\int Hd\gamma=\int Hd\hat{\gamma}$. This shows first that $\gamma$ and $\hat{\gamma}$ incur in the same cost under the semiseparability assumption, and second that the $y$-marginal of $\hat{\gamma}$ is $\nu$, which finally establishes $\hat{\gamma} \in \Pi_{cqm}(\mu,\nu)$. 
\begin{align*}\textstyle
\int Hd\hat{\gamma}=&\textstyle\,\,\int H \gamma(dx_1,dx_2,dy_1,dy_2)\gamma^{x_2,y_1,y_2}(dx_3,dy_3)\hat{\gamma}^{x_3,y_1,y_2,y_3}(dx_4,\dots,dy_4,\dots) \\
 = &\textstyle \textstyle \,\,\int H \gamma(dx_2,dy_1,dy_2)\gamma^{x_2,y_1,y_2}(dx_3,dy_3)\hat{\gamma}^{x_3,y_1,y_2,y_3}(dx_4,\dots,dy_4,\dots)\\=&\textstyle\,\, \int H \gamma(dx_2,dx_3,dy_1,dy_2,dy_3)\gamma^{x_3,y_1,y_2,y_3}(dx_4,dy_4)\hat{\gamma}^{x_4,y_1,y_2,y_3,y_4}(dx_5,\dots,dy_5,\dots)\\
  =& \textstyle\,\, \int H \gamma(dx_3,dy_1,dy_2,dy_3)\gamma^{x_3,y_1,y_2,y_3}(dx_4,dy_4)\hat{\gamma}^{x_4,y_1,y_2,y_3,y_4}(dx_5,\dots,dy_5,\dots)\\
  \dots \,\, =&\textstyle\,\,\int H \gamma(dx_{t-1},dy_1,\dots,dy_{t-1})\gamma^{x_{t-1},y_1,\dots,y_{t-1}}(dx_t,dy_t) \,\,=\,\, \int Hd\gamma. 
\end{align*}
\end{proof}

%\comment{Referee 2: Is causality linear/convex, and is this used by Lasalle? {\color{blue}JB: The answer is yes, and I suggest we only answer this (and that the difficulty is closedness, not convexity) in our reply}}
\begin{remark}
With Proposition \ref{prop:markov=cuasal} we can conclude that, whenever $\mu$ is Markov and the cost is l.s.c.\ and semiseparable, the problem \eqref{Pcqm} has a solution. Indeed, since \eqref{Pc}=\eqref{Pcqm} and by \cite[Corollary 1]{Lassalle2} the former is attained, one constructs also an optimizer for the latter. Our Theorem \ref{thm:markovDPP} yields the same with a stronger assumption in a self-contained way, through DPP. Let us stress that the causal quasi-Markov constraint is not a linear/convex one.
\end{remark}

All in all, the proof of Theorem \ref{thm:DPPcausal} is now trivial:

\begin{proof}[Proof of Theorem \ref{thm:DPPcausal}]
By Proposition \ref{prop:markov=cuasal} $value\eqref{Pcqm}= value\eqref{Pc}$ and by Theorem \ref{thm:markovDPP} we have the recursion and all the other properties.
\end{proof}

%{\color{blue}As promised in Section \ref{Main results}, we now illustrate the difficulty of the DPP with an example. 
%\begin{example}\label{Ex DPP} Take $N=2$ and $c=[x_1-y_1]^2+[x_2-y_2]^2$. Using the optimality of the monotone coupling on the line we get:
%\begin{align*}
%&\textstyle V_1(y_1;\gamma_1^{y_1}(dx_1)):= \inf\limits_{\gamma_2\in\Pi(\int_{x_1}\gamma_1^{y_1}(dx_1)\mu^{x_1},\nu^{y_1})}\int \gamma_2(dx_2,dy_2)[x_2-y_2]^2=\int_0^1\Bigl[F_{\nu^{y_1}}^{-1}(u)- F_{\int_{x_1}\gamma_1^{y_1}(dx_1)\mu^{x_1}}^{-1}(u)\Bigr]^2du,\\
%&\textstyle V_0 =  \inf\limits_{\gamma_1\in\Pi(p^1_*\mu,p^1_*\nu)}\Bigl \{ \int_{x_1,y_1} \gamma_1(dx_1,dy_1)[x_1-y_1]^2 + \int_{y_1}\nu(dy_1)\int_0^1\Bigl[F_{\nu^{y_1}}^{-1}(u)- F_{\int_{x_1}\gamma_1^{y_1}(dx_1)\mu^{x_1}}^{-1}(u)\Bigr]^2du\Bigr\}.
%\end{align*}
%From this the non-linear behaviour of the cost function in $V_0$, in terms of $\gamma_1$, is apparent.
%\end{example}
%}

We close the present part with a preliminary illustration of how the causal DPP can be used in practice; here we show a condition giving the equivalence between \eqref{Pc} and \eqref{Pbc}:

\begin{corollary}\label{coro:wass one}
Consider $N=2$ and a separable cost of the form $c=c_1(x_1,y_1)+|x_2-y_2|$, with $c_1$ l.s.c.\ and bounded from below. Suppose 
%further Assumption \ref{as:biweakcont} and 
that for every $z,y_1$ one of the following two cases holds
$$\textstyle \mbox{either }\left[\forall x_1:F_{\mu^{x_1}}(z) \geq F_{\nu^{y_1}}(z)\right ]\,\,\, \mbox{ or }\,\,\, \left[\forall x_1:F_{\mu^{x_1}}(z) \leq F_{\nu^{y_1}}(z)\right ], $$
Then \eqref{Pc} is equivalent to \eqref{Pbc}.
\end{corollary}

\begin{proof}
Borrowing from Theorem \ref{thm:DPPcausal}, call
$$\textstyle V_1(y_1,\gamma^{y_1})= \inf_{m\in \Pi (\int\gamma^{y_1}(dx_1)\mu^{x_1},\nu^{y_1})}\int |x_2-y_2|m(dx_2,dy_2),$$
which by e.g.\ \cite[eq.(2.48), p.75]{Villani} equals $\textstyle \int_z\left |F_{\int\gamma^{y_1}(dx_1)\mu^{x_1}}(z) -F_{\nu^{y_1}}(z)\right |dz$. But $F_{\int\gamma^{y_1}(dx_1)\mu^{x_1}}(z)=\int\gamma^{y_1}(dx_1)F_{\mu^{x_1}}(z) $, so by our technical assumption:
\begin{align}\textstyle
\int_z\left|F_{\int\gamma^{y_1}(dx_1)\mu^{x_1}}(z) -F_{\nu^{y_1}}(z)\right |dz &=\textstyle \int_z\left|\int_{x_1}\gamma^{y_1}(dx_1)F_{\mu^{x_1}}(z) -F_{\nu^{y_1}}(z)\right |dz \notag \\
&=\textstyle \int_z\int_{x_1}\gamma^{y_1}(dx_1)\left |F_{\mu^{x_1}}(z) -F_{\nu^{y_1}}(z)\right |dz ,\label{no-Jensen}
\end{align} 
and by Fubini's Theorem we get $\textstyle V_1(y_1,\gamma^{y_1})= \int_{x_1}\gamma^{y_1}(dx_1) \inf_{m\in\Pi(\mu^{x_1},\nu^{y_1})}\int|x_2-y_2|dm $ and again from Theorem \ref{thm:DPPcausal} we obtain that 
$$\textstyle Value\eqref{Pc}=\inf_{\gamma\in\Pi(p^1_*\mu,p^1_*\nu )}\int \gamma(dx_1,dy_1)\left\{ c_1(x_1,y_1)+\inf_{m\in\Pi(\mu^{x_1},\nu^{y_1})}\int|x_2-y_2|dm\right \},$$
which as we shall see in Proposition \ref{prop:dynprog}, equals the bicausal DPP, so we conclude.
\end{proof}
%
%\begin{remark}
%Notice that the condition on the conditional distributions in the previous result is satisfied if $\mu$ has independent components; it is easy to see that the condition is, however, much more general. At the moment we do not know how to extend this type of result to $N>2$ or to costs other than distance-type.
%\end{remark}

\section{The bicausal case}\label{Bicausal case}
%\comment{We might not give many proofs here, as they are too similar and easier than the quasi-Markov case}

We can obtain a similar DPP for bicausal transport plans. Let us introduce:
%
%\begin{theorem}
%Let $\mu, \nu \in \mathcal{P}(\R^N)$. If $\gamma\in \Pi_{bc}(\mu,\nu)$ is decomposed as in (\ref{successivedisintegration}), then it holds:
%%
\begin{multline}\label{Dyn-Pbc}\textstyle
\inf_{ \gamma^1 \in \Pi(p^1_*\mu,p^1_*\nu) } \int \gamma^1(dx_1,dy_1) \inf_{\gamma^2 \in \Pi(\mu^{x_1}, \nu^{y_1} )} \int \gamma^2(dx_2,dy_2)\dots \\ \textstyle
\dots \inf_{\gamma^{N} \in \Pi(\mu^{x_1,\dots,x_{N-1}}, \nu^{y_1,\dots,y_{N-1}} )} \int \gamma^{N}(dx_N,dy_N) c(x_1,\dots,x_{N},y_1,\dots,y_{N}) . \tag{Dyn-Pbc}
\end{multline}
%%
%\end{theorem}
The previous recursive problem is motivated by the following structure result, the proof of which is  analogous to the causal case, so we omit it:
\begin{proposition}\label{prop:recursivecharacterization}
Let $\mu,\nu\in\mathcal{P}(\R^N)$.\\
If $\gamma\in\Pi_{bc}(\mu,\nu)$ is decomposed as in \eqref{successivedisintegration}, then the following conditions on the kernels hold:
\begin{itemize}
\item[(i)] $\bar{\gamma}\in \Pi(p^1_*\mu,p^1_*\nu)$, and 
\item[(ii)] successively for $t<N$ and for $\gamma$-almost every $x_1,\dots,x_{t},y_1,\dots,y_{t}$ holds $$\textstyle \gamma^{x_1,\dots,x_{t},y_1,\dots,y_{t}}(dx_{t+1},dy_{t+1}) \in \Pi(  \mu^{x_1,\dots,x_{t}}(dx_{t+1}), \nu^{y_1,\dots,y_{t}}(dy_{t+1}) ) .$$
\end{itemize}
Conversely, given regular kernels $$\textstyle\bar{\gamma}(dx_1,dy_1),\gamma^{x_1,y_1}(dx_2,dy_2),\dots,\gamma^{x_1,\dots,x_{N-1},y_1,\dots,y_{N-1}}(dx_N,dy_N),$$ satisfying the properties $(i)-(ii)$, the measure $\gamma$ constructed as in \eqref{successivedisintegration} belongs to $\Pi_{bc}(\mu,\nu)$.
\end{proposition} 
The recursion corresponding to \eqref{Dyn-Pbc} (starting from $V_N^c:=c$) is:
\begin{multline}\textstyle
V_t^c(x_1,\dots,x_{t},y_1,\dots,y_{t})= \\\textstyle \inf_{\gamma^{t+1} \in \Pi(\mu^{x_1,\dots,x_{t}}, \nu^{y_1,\dots,y_{t}} )} \int \gamma^{t+1}(dx_{t+1},dy_{t+1}) V^c_{t+1}(x_1,\dots,x_{t+1},y_1,\dots,y_{t+1}), \label{valuefunctiongeneral}
\end{multline}
%
%and by the same token, is a l.s.c.\ function of its arguments. 
% Clearly the value of \eqref{Dyn-Pbc} (equal to \eqref{Pbc} and \eqref{Dbc}) then has the recursive representation: 
and so we want to compare the values of \eqref{Dyn-Pbc}, \eqref{Pbc} and
$$\textstyle V^c_0:=\inf_{ \gamma^1 \in \Pi(p^1_*\mu,p^1_*\nu) } \int V^c_1(x_1,y_1)\gamma^1(dx_1,dy_1).$$
%
%
%\comment{JB: Can we give a precise reference for Pflug and Pichler here too?}
We now give the DPP for the bicausal case, which appeared in \cite[Theorem 3]{Rueschendorf}; we will prove it with our methods, obtaining further attainability and some regularity of the value function. Observe that no Markovianity of $\mu$ or separability of the costs is needed, and that the above value function (and recursion) are more tractable than in the causal quasi-Markov case \eqref{weak-trans-cqm}.
\begin{proposition}\label{prop:dynprog}
Given a Borel bounded from below cost function $c$, we have that the recursive optimization problem \eqref{Dyn-Pbc} is well-defined, namely the successive integrals in \eqref{valuefunctiongeneral} are well-defined, and the values of \eqref{Dyn-Pbc}, \eqref{Pbc} and $V_0^c$ coincide. If further $c$ is l.s.c.\ and Assumption \ref{as:biweakcont} holds, then there is a bicausal optimizer and the value functions $V_t^c$ are all l.s.c.
\end{proposition}

\begin{proof}
Since $(x_1,\dots,x_t,y_1,\dots,y_t)\mapsto (\mu^{x_1,\dots,x_t},\nu^{y_1,\dots,y_t})$ is Borel and the set
$\{(p,q,\gamma):\gamma\in\Pi(p,q)\}$ is clearly closed, we get that 
$$\textstyle D_t:=\left\{(x_1,\dots,x_t,y_1,\dots,y_t,\gamma):\gamma\in\Pi(\mu^{x_1,\dots,x_t},\nu^{y_1,\dots,y_t}) \right \},$$
is Borel. As in the proof of Theorem \eqref{thm:markovDPP}, we start by proving recursively that the $V^c_t$'s are lower semianalytic (l.s.a). As in Step 2 therein one observes first that 
$$\textstyle (x_1,\dots,x_{N-1},y_1,\dots,y_{N-1},\gamma)\mapsto \int \gamma(dx_N,dy_N)c(x_1,\dots,x_{N},y_1,\dots,y_{N}),$$
is Borel, and so l.s.a. Since $V_{N-1}^c(x_1,\dots,x_{N-1},y_1,\dots,y_{N-1})$ is the infimum of this function over the fiber of $D_{N-1}$ at $(x_1,\dots,x_{N-1},y_1,\dots,y_{N-1})$ we get by \cite[Proposition 7.47]{BertsekasShreve} that  $V_{N-1}^c$ is l.s.a.\ and in particular universally measurable. The recursive step is obvious, and we get this result for each $V_{t}^c$, and so the integrals in \eqref{valuefunctiongeneral} are well-defined. Now take, as in Step 3 of the proof of Theorem \eqref{thm:markovDPP}, a universally measurable selection of $\varepsilon$-optimizers for $V_t^c$; call these $\gamma_{t,\varepsilon}^{x_1,\dots,x_{t},y_1,\dots,y_{t}} $. Build then $\Gamma_{\varepsilon}$ as the unique Borel measure that comes out of concatenating these (\cite[Proposition 7.45]{BertsekasShreve}). By construction each $\Gamma_{\varepsilon}$ is in $\Pi(\mu,\nu)$ (see Proposition \ref{prop:recursivecharacterization}) and $V_0^c+N\varepsilon\geq \int cd\Gamma_{\varepsilon}\geq value\eqref{Pbc}$ and so $V_0^c\geq value\eqref{Pbc}$. The reverse inequality follows trivially from Proposition \ref{prop:recursivecharacterization}.

If Assumption \ref{as:biweakcont} holds then $D_t$ is closed and so if further $c$ is l.s.c.\ we argue as in Step 5 of the proof of Theorem \eqref{thm:markovDPP}, proving that the value functions are l.s.c.\ as well. In turn, applying now \cite[Proposition 7.50(b)]{BertsekasShreve} we can obtain a universally measurable selection of optimizers $\gamma_t^{x_1,\dots,x_{t},y_1,\dots,y_{t}}$ and with them we build an optimal bicausal transport $\Gamma$ (just like the previous paragraph with $\varepsilon=0$).
\end{proof}

\medskip

Since at each step of the dynamic programming principle \eqref{Dyn-Pbc}, or equivalently \eqref{valuefunctiongeneral}, we have a usual optimal transport problem, we can write these in their dual formulation; for simplicity we assume that $c$ is l.s.c.\ and non-negative (see however \cite{MathiasWalter} for an extension). In effect, the value function at time $t$ for each $t=N-1,N-2,\dots,1$ is obtained recursively as:
%
%\begin{multline} 
%V^c(N-1;x_1,\dots,x_{N-1},y_1\dots,y_{N-1}):= \\
%\sup\limits_{ \substack{ \phi_N,\psi_N\in C_b(\R),\\ \phi_{N}(x_N) + \psi_{N}(y_N) \leq c(x_1,\dots,x_{N},y_1,\dots,y_{N}) }} \int \phi_N(x_N)  \mu^{x_1,\dots,x_{N-1}}(dx_N) + \int \psi_N(y_N)  \nu^{y_1,\dots,y_{N-1}}(dy_N), 
%\end{multline}
%
%
\begin{multline*} \textstyle
V^c(t;x_1,\dots,x_{t},y_1\dots,y_{t}):= \\ \textstyle
\sup\limits_{  \substack{ \phi_{t+1},\psi_{t+1}\in C_b(\R),\\ \phi_{t+1}(x_{t+1}) + \psi_{t+1}(y_{t+1}) \leq V^c(t+1;x_1,\dots,x_{t+1},y_1,\dots,y_{t+1}) } } \left\{ \int \phi_{t+1}(x_{t+1})  \mu^{x_1,\dots,x_{t}}(dx_{t+1}) \right .  \\ \textstyle
\left .+ \int \psi_{t+1}(y_{t+1})  \nu^{y_1,\dots,y_{t}}(dy_{t+1})\right \},
\end{multline*}
so the value of (\ref{Dyn-Pbc}) is also given by
$$\textstyle V^c(0):=\sup\limits_{ \substack{\phi_{1},\psi_{1}\in C_b(\R) \\ \phi_{1}(x_{1}) + \psi_{1}(y_{1}) \leq V^c(1;x_1,y_1) }} \int \phi_{1}(x_{1})  p^1_*\mu(dx_{1}) + \int \psi_{1}(y_{1}) p^1_*\nu(dy_{1}) .$$
Observe that this recursive structure of the dual problem is not obvious from \eqref{Dbc}, but can be guessed a posteriori thanks to the apparent ``primal'' recursive structure.
%An interesting case in applications is when $c$ has a separable (additive) structure, as:
%%
%\begin{align}
%c^h(x_1,\dots,x_{N},y_1,\dots,y_{N})&:= \sum_{t=1}^N h_t(x_t,y_t),\label{separablecost} \\
%c^p(x_1,\dots,x_{N},y_1,\dots,y_{N})&:= \sum_{t=1}^N |x_t-y_t|^p.
%%\\
%%\tilde{c}^h(x_1,\dots,x_{N},y_1,\dots,y_{N})&:=h_1(x_1,y_1)+ \sum_{t=2}^N h_t(x_t-x_{t-1},y_t-y_{t-1}), \\
%%\tilde{c}^p(x_1,\dots,x_{N},y_1,\dots,y_{N})&:= |x_1-y_1|^p+ \sum_{t=2}^N |(x_t-x_{t-1})-(y_t-y_{t-1})|^p.
%\end{align}
%%
%For such costs, we will always assume that $h_t(\cdot)$ is bounded from below. \\
%In the case of $c^h$, and so for $c^p$, the problem (and the recursion \eqref{valuefunctiongeneral}) reduces considerably.
A simpler picture of \eqref{Dyn-Pbc} arises when $c$ has separable structure as in \eqref{separable cost}; see \cite[Theorem 7.2]{PflugPichler}.\\

%we leave the details to the reader. \\

We give now the belated proof of Theorem \ref{main theorem}:

\begin{proof}[Proof of Theorem \ref{main theorem}] The result follows from Proposition \ref{main result bc} below, if the causal/bicausal equality is proved. {We stress that Condition \eqref{Ruesch+indep cond}, which is needed for Proposition \ref{main result bc}, is satisfied in the present case because the conditional distribution functions associated to $\mu$ do not depend on the conditioning argument}. We now prove the causal/bicausal equality. Start with $\gamma \in \Pi_c(\mu,\nu)$ and decompose it as in (\ref{successivedisintegration}). From Proposition \ref{prop:equiv charact}, we know that the following conditions ($t<N$) are satisfied by the kernels $ \bar{\gamma}(dx_1,dy_1), \gamma^{x_1,y_1}(dx_2,dy_2)$, up to $   \gamma^{x_1,\dots,x_{N-1},y_1,\dots,y_{N-1}}(dx_N,dy_N)$:
\begin{align}\label{first marginal}\textstyle
\gamma^{x_1,\dots,x_{t},y_1,\dots,y_{t}}(dx_{t+1}, \R) = \mu_{t+1}(dx_{t+1})
\end{align}
and
\begin{align}\label{second marginal}\textstyle
\int_{x_1,\dots,x_{t}} \gamma^{x_1,\dots,x_{t},y_1,\dots,y_{t}}(\R,dy_{t+1})\gamma^{y_1,\dots,y_{t}}(dx_1,\dots,dx_{t}) \,\,\, \,\, & = \,\,\, \nu^{y_1,\dots,y_{t}}(dy_{t+1}) \\ \nonumber
& = \gamma^{y_1,\dots,y_{t}}(\R, dy_{t+1}) 
\end{align}
We can rewrite (\ref{first marginal}) in the following way:
\begin{align}\label{modified first marginal}\textstyle
\int_{x_1,\dots,x_{t}} \gamma^{x_1,\dots,x_{t},y_1,\dots,y_{t}}(dx_{t+1}, \R)\gamma^{y_1,\dots,y_{t}}(dx_1,\dots,dx_{t}) \,\,\, \,\, & =\,\,\, \mu_{t+1}(dx_{t+1}) \\  \nonumber
& = \gamma^{y_1,\dots,y_{t}}(dx_{t+1},\R) 
\end{align}
Therefore, we can construct a new plan $\tilde{\gamma}$ as follows
\begin{equation*}\textstyle
\tilde{\gamma}(dx_1,\dots,dx_N,dy_1,\dots,dy_N)=\bar{\gamma}(dx_1,dy_1)\gamma^{y_1}(dx_2,dy_2)\dots \gamma^{y_1,\dots,y_{N-1}}(dx_N,dy_N), 
\end{equation*}
with
\begin{align}\label{kernels of gamma tilde}\textstyle
\gamma^{y_1,\dots,y_t} (dx_{t+1}, dy_{t+1}) = \int_{x_1,\dots,x_{t}} \gamma^{x_1,\dots,x_{t},y_1,\dots,y_{t}}(dx_{t+1}, dy_{t+1} )\gamma^{y_1,\dots,y_{t}}(dx_1,\dots,dx_{t}).
\end{align}
Due to (\ref{modified first marginal}) and (\ref{second marginal}), one can see that for any $t<N$ each kernel 
$$\textstyle \gamma^{y_1,\dots,y_t} \in \Pi(\mu_{t+1}(dx_{t+1}),\nu^{y_1,\dots,y_{t}}(dy_{t+1}) ). $$
Then, from Proposition (\ref{prop:recursivecharacterization}), we know that $\tilde{\gamma} \in \Pi_{bc}(\mu,\nu)$. Writing down the minimization over the kernels (\ref{kernels of gamma tilde}), we have exactly (\ref{Pbc}). Since the kernels of $\tilde{\gamma}$ and $\gamma$ are connected via (\ref{modified first marginal}) and (\ref{second marginal}), one can observe that the value of the bicausal transport problem is less or equal than the causal one, giving us the desired result.
\end{proof}

\begin{proof}[Proof of Corollary \ref{coro KR}]
Follows from Theorem \ref{main theorem} by observing that the map $$\textstyle (x_1,\dots,x_N,y_1,\dots,y_N)\mapsto (x_1,x_2-x_1,\dots,x_N-x_{N-1},y_1,y_2-y_1,\dots,y_N-y_{N-1}),$$
preserves causality when applied to $\gamma\in \Pi_c(\mu,\nu)$.
\end{proof}
\medskip

The proof of Theorem \ref{main theorem} rests on the following result, which needs Condition \eqref{Ruesch+indep cond} encompassing at the same time the independence condition \eqref{independence cond} and R\"uschendorf's  ``monotone regression dependence'' in \cite[Corollary 2]{Rueschendorf}. %We also recall the general definition of Knothe-Rosenblatt rearrangement(extending the one in Theorem \ref{main theorem}).
%
%\begin{condition}\label{Ruesch+indep cond} 
%For any $t=1,\dots, N-1$, for any vectors $(x_1,\dots,x_t) \leq (\bar{x}_1, \dots, \bar{x}_t)$ and $(y_1,\dots,y_t) \leq (\bar{y}_1, \dots, \bar{y}_t)$ the following holds for every $u$:
%\begin{align}
%\left( F_{\mu^{x_1,\dots,x_t}}(u) - F_{\mu^{\bar{x}_1, \dots, \bar{x}_t}}(u) \right)  \left( G_{\nu^{y_1,\dots,y_t}}(u) - G_{\nu^{\bar{y}_1, \dots, \bar{y}_t}}(u) \right) \geq 0 .
%For all $t=1,\dots,N-2$, all $(x_1,\dots,x_t ),(y_1,\dots,y_t )$, and all $u$ holds:
%\begin{align}
%\left( F_{\mu^{x_1,\dots,x_t,\bar{x}}}(u) - F_{\mu^{{x}_1, \dots, {x}_t,x}}(u) \right)  \left( G_{\nu^{y_1,\dots,y_t,\bar{y}}}(u) - G_{\nu^{{y}_1, \dots, {y}_t,y}}(u) \right) \geq 0 %,
%\end{align}
%whenever $\bar{x}\geq x$ and $\bar{y}\geq y$.
%\end{condition}
%
%Th condition encompasses the first one, as well as R\"uschendorf's condition of ``monotone regression dependence'' in \cite[Corollary 2]{Rueschendorf}. %The following proposition yields the existence of an optimal Monge map for the bicausal problem.  

%\comment{Referee 2: where does $u$ live? Done.}
\begin{proposition}
\label{main result bc}
For each $t=1,\dots,N-2$, all $(x_1,\dots,x_t ),(y_1,\dots,y_t )$, and {$u\in\mathbb{R}$}, suppose
%\comment{Referee 2: who is $G$? Done (it was F)}
\begin{align}\label{Ruesch+indep cond} 
\left( F_{\mu^{x_1,\dots,x_t,\bar{x}}}(u) - F_{\mu^{{x}_1, \dots, {x}_t,x}}(u) \right)  \left( F_{\nu^{y_1,\dots,y_t,\bar{y}}}(u) - F_{\nu^{{y}_1, \dots, {y}_t,y}}(u) \right) \geq 0 ,
\end{align}
whenever $\bar{x}\geq x$ and $\bar{y}\geq y$. Assume further that
\begin{align}\notag\textstyle
c(x_1,\dots,x_{N},y_1,\dots,y_{N})&:= \sum_{t\leq N} c_t(x_t-y_t),
\end{align}
where each $c_t$ is convex. Then a solution to \eqref{Pbc} is given by the Knothe-Rosenblatt rearrangement \eqref{quantile transforms}.
%:
%%
%\begin{align}\label{quantile transforms}
%& X_1^* = F_{\mu_1}^{-1} (U_1),\hspace{46pt}  Y_1^* = F_{\nu_1}^{-1} (U_1), \hspace{31pt}\mbox{ and inductively }\\ \nonumber
%& X_n^* = F_{\mu^{X_1^*,\dots,X_{n-1}^*}}^{-1} (U_n),\quad Y_n^* = F_{\nu^{Y_1^*,\dots,Y_{n-1}^*}}^{-1} (U_n),\, \text{for } n=2,\dots,N,
%\end{align}
%%
%where $U_1,\dots,U_N$ are independent and uniformly distributed random variables on $[0,1]$. 
Additionally, if $\mu$-a.s.\ all the conditional distributions of $\mu$ are atomless (e.g.\ if $\mu$ has a density), then this rearrangement is induced by the Monge map determined by \eqref{KRdef}.
%
% 
%$$T(x_1,\dots,x_N)=(T^1(x_1),\dots,T^N(x_N; x_1, \dots. x_{N-1}))$$
%
%which is defined as follows: 
%
%\begin{align}
%T^1(x_1)&:= \,\, F_{\nu_1}^{-1}\circ F_{\mu_1}(x_1), \nonumber\\
%T^n(x_n; x_1,\dots,x_{n-1})& :=\,\, F_{\nu^{T^1(x_1),\dots,T^{n-1}(x_{n-1};x_1,\dots,x_{n-2})}}^{-1}\circ F_{\mu^{x_1,\dots,x_{n-1}}}(x_n)\,\, \label{KRdef}
%\end{align}
%%
%for $n=2,\dots,N$. 
%The optimal value of the bicausal problem is:
%%
%\begin{align}
%V^h(0) = \sum\limits_{t=1}^N h_t \left( x_t - T^t(x_t; x_1, \dots. x_{t-1}) \right),
%\end{align}
%% 
%where for $t=1$ $T^t(x_t; x_1, \dots. x_{t-1})$  is just $T^1(x_1)$.

%\item Under Condition (\ref{independence cond}) the coupling in \eqref{quantile transforms} is optimal for \eqref{Pc}, with $X_n^* = F_{\mu_n}^{-1} (U_n)$ of course, and if each $\mu_i$ is atomless (which holds if they have a density), then the map in \eqref{KRdef} is optimal, with $$T^n(x_n; x_1,\dots,x_{n-1}) :=\,\, F_{\nu^{T^1(x_1),\dots,T^{n-1}(x_{n-1};x_1,\dots,x_{n-2})}}^{-1}\circ F_{\mu_n}(x_n).$$
%\end{enumerate}
\end{proposition}

\begin{proof}
First, we give the proof for the case when the source measure is the product of its marginals, for $N=2$. From Proposition \ref{prop:dynprog} we know that the values of \eqref{Pbc} and \eqref{Dyn-Pbc} coincide, and by assumption $\mu(dx_1,dx_2) = \mu_1(dx_1)\mu_2(dx_2)$, so \eqref{Dyn-Pbc} has the form
\begin{multline}\label{eq:Dyn-Pbc n=2}\textstyle
\inf_{ \gamma^1 \in \Pi(\mu_1,p^1_*\nu) } \int \gamma^1(dx_1,dy_1) \left[ c_1(x_1-y_1) + \inf_{\gamma^2 \in \Pi(\mu_2, \nu^{y_1} )} \int c_2(x_2-y_2) \gamma^2(dx_2,dy_2)  \right]. 
\end{multline}
From classical optimal transport (see \cite{Villani}) it is known that 
$$\textstyle \inf_{\gamma^2 \in \Pi(\mu_2, \nu^{y_1} )} \int c_2(x_2-y_2) \gamma^2(dx_2,dy_2) = \int c(F_{\mu_2}^{-1} (u_2)- F_{\nu^{y_1}}^{-1} (u_2)) d u_2. $$
The last value does not depend on $x_1$, therefore, it is constant for the minimization with respect to $\gamma^1$ in (\ref{eq:Dyn-Pbc n=2}), and we conclude that the pair $(X_1^*,Y_1^*)$ will be the optimizer at this step, which in turn allow to construct $(X_2^*,Y_2^*)$ by the previous considerations. For general $N$ the result follows by induction, iterating the arguments so far. When $\mu$ is not the product of its marginals, yet the monotone regression assumption of \cite{Rueschendorf} holds, the proof of this result is given in \cite[Corollary 2]{Rueschendorf}. Finally, it is easy to observe that Condition \ref{Ruesch+indep cond} unifies exactly the two cases described.
%The proof of this is much as given in \cite[Corollary 2]{Rueschendorf}, under the author's condition. However, one can easily observe that it remains valid under our Condition \ref{Ruesch+indep cond}. 
\end{proof}
\medskip

We stress that in case the independent marginals condition \ref{independence cond} does not hold, and even if condition \ref{Ruesch+indep cond} is true, Example \ref{counerexample} shows that causal and bicausal values may differ.

\medskip
%We stress that existence of a Monge solution (which again is given by the Knothe-Rosenblatt rearrangement) for the bicausal problem had already been obtained in \cite[Corollary 2]{Rueschendorf} under the so-called ``monotone regression dependence'' assumption. For $N=2$ this assumption means that $x\mapsto F_{\mu^x}(t)$ and $y\mapsto F_{\nu^y}(t)$ are, for each given $t$, simultaneously either decreasing or increasing. In Proposition \ref{prop:furtherproperties}.(3) we required nothing on $y\mapsto F_{\nu^y}(t)$, and so monotone regression dependence need not hold.

The explicit form of the optimizers in Theorem \ref{main theorem} is obviously only true for $\R$-valued processes. Interestingly, if the transport problem consisted in mapping (in a bicausal way) $\R^L$-valued process, the recursive structure is just as in the case $L=1$ and so if $\mu$ were the product of its marginals (each of them in $\mathcal{P}(\R^L)$ now) we would get a similar conclusion with the role of the monotone rearrangements taken by general Brenier maps, under suitable conditions.

%\comment{Referee 2: do not understand nor believe it. {\color{blue}JB: I expanded the remark. TO DO: add more material and craft careful answer}}
\begin{remark}\label{rem:second_Monge}
%We continue the discussion in Remark \ref{rem:first_Monge}. Assuming that $T:\R^N\to\R^N$ is measurable and adapted, then the bicausality criterion given there holds as soon as the maps $T^s(x_1,\dots,x_s)$ were built through $T^s=M^s(T^1(x_1),\dots,T^{s-1}(x_1,\dots,x_{s-1}),x_s)$, for $M^s:\R^s\to\R$ Borel-measurable having a Borel left-inverse w.r.t.\ the last coordinate (the other ones fixed). This the case of the 
We reassure the reader that the Knothe-Rosenblatt rearrangements in the form \eqref{quantile transforms} are always bicausal, and in the Monge form \eqref{KRdef} this is also the case as soon as all conditional distributions of the source measure are atomless. We illustrate the argument for \eqref{quantile transforms} as follows: {for each bounded Borel $g(\cdot,\cdot)$ we may define $$\textstyle y_1\mapsto  G(y_1):=\int_0^1 g(y_1,F^{-1}_{\nu^{y_1}}(v))dv,$$ so that denoting $X_1^*:=F^{-1}_{\mu_1}(U_1)$, $Y_1^*:=F^{-1}_{\nu_1}(U_1)$ and $Y_2^*:=F^{-1}_{\nu^{F^{-1}_{\nu_1}(U_1)}}(U_2)$, we get %by Fubini Theorem
\begin{align*}\textstyle
E[f(X_1^*)g(Y_1^*,Y_2^*)]&= \textstyle
\int_0^1\int_0^1 f(F^{-1}_{\mu_1}(u_1))g\Bigl(F^{-1}_{\nu_1}(u_1),F^{-1}_{\nu^{F^{-1}_{\nu_1}(u_1)}}(u_2) \Bigr )du_2du_1\\ &=\textstyle
\int_0^1\int_0^1 f(F^{-1}_{\mu_1}(u_1)) G\left(F^{-1}_{\nu_1}(u_1)\right ) du_1\\
&=\textstyle E[\, E[f(X_1^*)|Y_1^*]\,G(Y_1^*) ]\\
&=\textstyle E[\, E[f(X_1^*)|Y_1^*]\,g(Y_1^*,Y_2^*) ].
\end{align*}
Thus the law of $X_1^*$ given $(Y_1^*,Y_2^*)$, equals the law of $X_1^*$ given $Y_1^*$.} The same holds inverting the roles of $\mu$ and $\nu$ and going to greater time indices. As for \eqref{KRdef}, the argument actually follows directly upon noticing for example that for $\mu-$a.e.\ $x_1$ the measure $[F_{\mu^{x_1}}]_{*}\mu^{x_1}$ is equal to Lebesgue measure on $[0,1]$, under the given assumptions. {Even though these considerations are not explicit in \cite{Rueschendorf}, they underpin some of the results therein.}
%, then immediately $\gamma^T$ is bicausal as soon as the $x$'s (marginals of $\mu$) are independent; such is the case of Euler-type discretizations of stochastic differential equations driven by a noise process with independent increments.
\end{remark}

{
\subsection{A characterization of the Knothe-Rosenblatt rearrangement}
\label{sec Bogashev}

As we have seen, the Knothe-Rosenblatt rearrangement \eqref{quantile transforms}-\eqref{KRdef} appears quite naturally in our setting. In light of Remark \ref{rem:second_Monge}, we would like to characterize it as the unique bicausal transport plan with a desirable ``increasingness'' property. The correct concept turns out to be that of increasing triangular transformations, found in \cite{triangular}, which we recall:

\begin{definition}
A map $T:\R^N\to \R^N$ is an increasing triangular transformations (in short ITT) with source $\mu$ if for each $t\in\{1,\dots,N\}$ there is a function $$\R^t\ni (x_1,\dots,x_t)\mapsto T^t(x_t;x_1,\dots,x_{t-1})\in\R,$$ such that $T^t(\cdot\,;x_1,\dots,x_{t-1})$ is non-decreasing\footnote{For $t=1$ one should understand $T^1$ as a non-decreasing function of $x_1$ only.} for $\mu$-almost every $(x_1,\dots,x_{t-1})$ and further
$$T(x_1,\dots,x_N)\, =\, (T^1(x_1),T^2(x_2;x_1),\dots,T^N(x_N;x_1,\dots,x_{N-1})) ,\,\,\, \mu-a.s.$$
\end{definition}

Clearly the Knothe-Rosenblatt map \eqref{KRdef} is an ITT. As \cite[Lemma 2.1]{triangular} shows, given that all conditional distributions of $\mu$ and $\nu$ are atomless, there is a canonical ITT that pushes $\mu$ into $\nu$ which is unique up to $\mu$-negligible sets: the Knothe-Rosenblatt map. Under the same conditions on $\mu$ and $\nu$, this characterization has a geometric counterpart; as shown in e.g. \cite[Chapter 2.3, Remark 2.20]{Santambrogiobook}, the Knothe-Rosenblatt map is the unique map which is increasing in the lexicographical order of $\R^N$ and pushes $\mu$ into $\nu$. Under the same assumptions \cite{CGS_KR_to_Brenier} obtains the Knothe-Rosenblatt map as a natural limit of Brenier maps\footnote{Beware that \cite{Santambrogiobook,CGS_KR_to_Brenier} read coordinates from $N$ to $1$ rather than from $1$ to $N$ as we do.}. We ask: what distinguishes/characterizes the Knothe-Rosenblatt map without any assumption on $\nu$? Here is the first result:

\begin{proposition}
\label{prop charact KR}
Suppose all conditional distributions of $\mu$ are atomless. Then the Knothe-Rosenblatt map from $\mu$ into $\nu$, defined in \eqref{KRdef}, is the unique bicausal transport plan between $\mu$ and $\nu$ which is induced by an increasing triangular transformation.
\end{proposition}

\begin{proof}
Let $\gamma$ be a bicausal transport plan between $\mu$ and $\nu$. By Proposition \ref{prop:recursivecharacterization}.(ii) we know that $\gamma^{x_1,\dots,x_{t-1},y_1,\dots,y_{t-1}}\in\Pi(\mu^{x_1,\dots,x_{t-1}},\nu^{y_1,\dots,y_{t-1}} )$. But if $\gamma$ is further induced by an ITT (say $T$) this means that $T^t(\cdot; x_1,\dots,x_{t-1})$ is pushing $\mu^{x_1,\dots,x_{t-1}}$ into $\nu^{y_1,\dots,y_{t-1}}$ in an increasing way (of course, under the understanding that $(y_1,\dots,y_{t-1})=(T^1,\dots,T^{t-1})(x_1,\dots,x_{t-1})$). This immediately shows that $T^1 = F^{-1}_{\nu_1}\circ F_{\mu_1}$ but for a $\mu$-null set. A straightforward induction argument proves that $\mu$-a.s.\ the Knothe-Rosenblatt map and $T$ are equal.
\end{proof}

The relevance of this result is that, as the counter-example after the proof of \cite[Lemma 2.1]{triangular} reveals, when $\nu$ has atoms in its conditional distributions then there may be many ITT's from $\mu$ into $\nu$. The same counter-example proves that the Knothe-Rosenblatt map need not be increasing in lexicographical order in this case, so it cannot be characterized in terms of lexicographical order in this generality. Furthermore, under this pathology of $\nu$, the Knothe-rosenblatt map need not be the natural limit of Brenier maps (see \cite[Example 2.26]{Santambrogiobook}). So, as far as we know, Proposition \ref{prop charact KR} is the only robust characterization of the Knothe-Rosenblatt map. We finally stress that this result can be extended to the case where $\mu$ and $\nu$ are arbitrary (so the Knothe-Rosenblatt rearrangement \eqref{quantile transforms} need not be induced by a map) provided one generalizes the definition of ITT to transport plans:

\begin{definition}
$\gamma\in\P(\R^N\times\R^N)$ is an increasing triangular transport if $\gamma(dx_1,dy_1)$, as well as $\gamma^{x_1,\dots,x_{t-1},y_1,\dots,y_{t-1}}(dx_t,dy_t)$ for $\gamma$-almost every $(x_1,\dots,x_{t-1},y_1,\dots,y_{t-1})$, all have monotone support as bivariate measures.\footnote{A set $\Gamma\subset \R^2$ is called monotone if $(x,y),(\bar{x},\bar{y})\in\Gamma$ and $x<\bar{x}$ implies $y\leq\bar{y}$. }
\end{definition}

\begin{proposition}
\label{prop charact KR general}
The Knothe-Rosenblatt rearrangement from $\mu$ into $\nu$, defined in \eqref{quantile transforms}, is the unique bicausal transport plan between $\mu$ and $\nu$ which is an increasing triangular transport.
\end{proposition}

The proof is essentially the same as in Proposition   \ref{prop charact KR}, so we omit it.
}

%\comment{Referee 1: there is no relation with rest of paper, and entropy only relevant for Gaussians {\color{blue}JB: I recommend we stress more that what we want here is to show the DPP at work ... and consider more general product measure to de-emphasize the role of Gaussians}}
\subsection{Digression into a classical functional inequality}\label{sec Talagrand}

Another instance where Theorem \ref{main theorem} (and Condition \ref{independence cond}) comes in handy, is the following interpretation of the proof of Talagrand's $\mathcal{T}_2$ inequality. Out point is to show how the ideas discussed in this article are pertinent to several fields in mathematics. First, recall (see \cite{Talagrand}): given a unit standard Gaussian measure $G^N$ and any other measure $\nu$ on $\R^N$, the following holds:
\begin{equation}\textstyle
2 Ent(\nu|G^N)\geq \mathcal{T}_2(G^N,\nu),\label{T2}\tag{T2}
\end{equation}  
where $Ent$ denotes relative entropy and $\mathcal{T}_2(\cdot,\cdot)$ is the value of the optimal transport problem with quadratic cost. Equality holds iff $\nu$ is an affine translate of $G^N$. We establish here the related bicausal inequality  
\begin{equation}\textstyle
2 Ent(\nu|G^N)\geq \mathcal{T}_{bc,2}(G^N,\nu), \mbox{ for all }\nu\in\mathcal{P}(\R^N),\label{CT2}\tag{CT2}
\end{equation}
where $ \mathcal{T}_{bc,2}(G^N,\nu)$ is the value of the optimal bicausal transport problem for quadratic cost. This clearly implies \eqref{T2} for $G^N$. It is the DPP for the bicausal transport problem that replaces the role of the usual tensorization trick; strictly speaking, the DPP is just giving a name to an intermediate step in Talagrand's well-known proof. The proof given here applies of course to other product measures, and in that setting is reminiscent of \cite[Proposition 22.5]{Villani_Old}; we stick to the gaussian case only for the sake of concreteness and because this relates to the Wiener case in continuous time. %It may be interesting to see if this method allows for shorter or stronger functional inequalities for some Markov Chains.
\begin{proposition}\label{prop:functional inequality}
 For $G^N$ the unit standard Gaussian measure on $\R^N$, the bicausal transport-entropy inequality  \eqref{CT2} holds.
\end{proposition}
\begin{proof}
 We start assuming \eqref{CT2} for $N=1$, which holds by \cite{Talagrand}, and prove it for $N=2$. The general inductive argument is then obvious and therefore skipped. By e.g.\ \cite[Lemma 10.3]{Varadhan} it is clear that for any $\nu\in\mathcal{P}(\R^2)$ holds:
\begin{align*}\textstyle
 Ent(\nu|G^2)&=Ent(\nu|G^1\otimes G^1) \\ & \textstyle =Ent(p^1_*\nu|G^1)+\int Ent(\nu^{y_1}| G^1) p^1_*\nu(dy_1) \geq \frac{\mathcal{T}_2(G^1,p^1_*\nu)}{2} +\int \frac{ \mathcal{T}_2(G^1,\nu^{y_1}) }{2} p^1_*\nu(dy_1) , 
\end{align*}
On the other hand, by the bicausal dynamic programming principle (Proposition \ref{prop:dynprog}) we get
\begin{align*}\textstyle
\mathcal{T}_{bc,2}(G^2,\nu) &= \textstyle\inf_{\gamma\in\Pi(G^1,p^1_*\nu)} \int \left\{|x_1-y_1|^2 +\mathcal{T}_2(G^1,\nu^{y_1}) \right\} \gamma(dx_1,dy_1)\\ &\textstyle = \mathcal{T}_2(G^1,p^1_*\nu) + \int\mathcal{T}_2(G^1,\nu^{y_1})p^1_*\nu(dy_1).
\end{align*} \end{proof}
\begin{remark}
It is clear that equality cannot generally hold in \eqref{CT2}, since for $N=1$ we have $\mathcal{T}_2(G,\nu)=\mathcal{T}_{bc,2}(G,\nu)$ and equality in \eqref{T2} is not usually the case. On the other hand, if $\nu$ is an affine translate of $G^N$, then of course $2 Ent(\nu|G^N)= \mathcal{T}_2(G^N,\nu)=\mathcal{T}_{bc,2}(G^N,\nu) $. More generally, there is equality in \eqref{CT2} provided $\nu^{y_1,\dots,y_t}(dy_{t+1})$ is an affine translate of $G^1$, for each $t$ and $\nu$-a.e.\ $y$. 
\end{remark}

%INSERT HERE CONNECTION BETWEEN FUNCTIONAL INEQUALITIES AND PFLUG ET AL.

\section{Some geometrical aspects of the bicausal case and connections with stochastic programming}
\label{lex_section}

At the end of this section we shall prove Theorem \ref{thm: funct nested}. We start however with a discussion about geometric properties of the space of probability measures endowed with a bicausal Wasserstein distance (equiv.\ nested distance). We notice first that the Knothe-Rosenblatt rearrangement offers a way to interpolate in a meaningful and non-linear way between stochastic programs. From Remark \ref{rem:second_Monge} we know that this rearrangement is bicausal, and as discussed in the previous section if all conditional probabilities $\mu^{x_1,\dots,x_n},\nu^{y_1,\dots,y_n}$ are atomless, then it is induced by a bicausal map which is characterized as the $\mu$-a.s.\ unique transformation increasing w.r.t.\ lexicographical order and which pushes forward $\mu$ onto $\nu$. Inspired by the concept of displacement interpolation/convexity in optimal transport (as in \cite[Chapter 5]{Villani}) let us define:%For the following result we first define what we call the \textit{lexicographical displacement interpolation} between $\mu$ and $\nu$:
%
%\comment{JB: Is this name cool enough? }
\begin{definition}
Let $\pi=\pi(\mu,\nu)$ be the Knothe-Rosenblatt rearrangement as in \eqref{quantile transforms}. The lexicographical displacement interpolation between $\mu$ and $\nu$ is then defined as the function
\begin{equation}
t\in [0,1]\mapsto [\mu,\nu]_t:= [(x,y)\mapsto(1-t)x+ty]_{*}\pi\in \mathcal{P}(\R^N), %([1-t]id+t T)_{*}\mu\in \mathcal{P}(\R^N).
\label{lex_interpolation}
\end{equation}
If all conditional probabilities $\mu^{x_1,\dots,x_n}$ are atomless, we also have 
$$[\mu,\nu]_t = ([1-t]id+t T)_{*}\mu,$$
where $T=T(\mu,\nu)$ is the Knothe-Rosenblatt map defined in \eqref{KRdef} and $id$ denotes the identity map in $\R^N$.
\end{definition}

Thus we have that $[\mu,\nu]_0=\mu$ and $[\mu,\nu]_1=\nu$; therefore the name. We prove now that many stochastic optimization problems are concave along the curves given by \eqref{lex_interpolation}. Recall from \eqref{def stoch prog} that the value of a stochastic program with cost $H$ and noise distribution $\eta\in\mathcal{P}(\R^N) $ is given by
$$\textstyle v(\eta):=\inf_{u_1(\dot),\dots, u_N()} \int H(x_1,\dots,x_N,u_1(x_1),u_2(x_1,x_2),\dots,u_N(x_1,\dots,x_N))\eta(dx) .$$

%\comment{JB: Does this lex-discplacement concavity property characterize the K-R map?}
\begin{theorem}\label{lexico_concavity}
Let $H:\R^N\times \R^N \to\R$ be bounded from below, concave in the first variable while convex in the second one. 
%Define the multistage (non-anticipative) stochastic program with cost $H$ and noise distribution $\eta\in\mathcal{P}(\R^N) $ by
%
%$$v(\eta):=\inf_{u_1(\dot),\dots, u_N(\cdot)} \int H(x_1,\dots,x_N,u_1(x_1),u_2(x_1,x_2),\dots,u_N(x_1,\dots,x_N))\eta(dx) .$$
%
Then the functions $$\textstyle t\in[0,1]\mapsto v([\mu,\nu]_t)$$ are concave for each $\nu$ and $\mu$ such that all conditional probabilities $\mu^{x_1,\dots,x_n}$ are atomless, so we may say that $v(\cdot)$ is lexicographic-displacement concave. If further $v(\mu)$ is attained,
then 
$$\textstyle v(\nu)\leq v(\mu) + \inf_{\substack{ u^*\in \argmin(v(\mu))}} \int \inf_{\xi\in\partial_x H(x,u^*(x))}\xi\cdot[T(x)-x]\mu(dx).$$
\end{theorem}

The previous result should be seen as a complement to the results of G.\ Pflug and A.\ Pichler \cite{Pflug,PflugPichler}, which give conditions under which  $|v(\nu)- v(\mu)|$ can be gauged by the value of a bicausal problem between $\mu$ and $\nu$. Indeed, Theorem \ref{lexico_concavity} highlights the connection between the most eminent of bicausal maps, i.e.\ the Knothe-Rosenblatt rearrangement, and multistage stochastic programming. On the other hand, the previous result can be related to \cite[Open Problem 5.17]{Villani}, with the caveat that we replaced the role of the Brenier's map by the Knothe-Rosenblatt rearrangement when defining the interpolations. %The last part of the above result can be compared to \cite[Proposition 5.29]{Villani}, had we further taken for granted a suitable uniform concavity assumption.\\

%We turn to the proof of Theorem \ref{lexico_concavity}, so the reader should recall the definition of lexicographic displacement interpolation given in \eqref{lex_interpolation}.\\

\begin{proof}[Proof of Theorem \ref{lexico_concavity}] Take $\mu,\nu,H$ as stated, and $a,b,s,t\in [0,1]$ with $a+b=1$. We will show that $v([\mu,\nu]_{at+bs})\geq av([\mu,\nu]_{t})+bv([\mu,\nu]_{s}) $. We write
\begin{align*}\textstyle
A_n(x)&=(1-t)x_n+tT^n(x_n;x_1,\dots,x_{n-1}) \\
B_n(x)&=(1-s)x_n+sT^n(x_n;x_1,\dots,x_{n-1})\\
C_n(x)&=(1-at-bs)x_n+[at+bs]T^n(x_n;x_1,\dots,x_{n-1}),
\end{align*}
with $x = (x_1,\dots,x_n),$
so we have
\begin{multline*}\textstyle
v([\mu,\nu]_{at+bs})=
\inf_{u_1(),\dots, u_N()} \int H\Bigl(C_1(x),\dots,C_N(x),
u_1\bigl(C_1(x)\bigr),\dots,u_N\bigl(C_1(x),\dots,C_N(x)\bigr)\Bigr)\mu(dx), 
\end{multline*}
and by the concavity assumption,
\begin{multline}\label{ineq:concave val fun}\textstyle
v([\mu,\nu]_{at+bs})\geq\\  \textstyle
a\left\{\inf_{u_1(\dot),\dots, u_N(\cdot)} \int H\Bigl(A_1(x),\dots,A_N(x),
u_1\bigl(C_1(x)\bigr),\dots,u_N\bigl(C_1(x),\dots,C_N(x)\bigr)\Bigr)\mu(dx)\right\} \\ \textstyle
+b \left\{\inf_{u_1(\dot),\dots, u_N(\cdot)} \int H\Bigl(B_1(x),\dots,B_N(x),
u_1\bigl(C_1(x)\bigr),\dots,u_N\bigl(C_1(x),\dots,C_N(x)\bigr)\Bigr)\mu(dx)\right\}. 
\end{multline}
If say $t<1$, we have that $A_1(x) = A_1(x_1)$ is injective (as it is strictly increasing) so defining $\tilde{u}_1(\cdot)=u_1\circ C_1\circ A_1^{-1}(\cdot)$ we have $\tilde{u}_1(A_1(x))=u_1(C_1(x))$. For $n=2$ we introduce $$\textstyle \tilde{u}_2(z_1,z_2)=u_2\bigl(\,\,C_1\circ A_1^{-1}(z_1)\,\, ,\,\, C_2\bigl(A_1^{-1}(z_1), A_2^{-1}(z_2|A_1^{-1}(z_1)) \bigr)\,\,\bigr),$$ which is well-defined due to the function $A_2$ being strictly increasing in its second variable, and verify that $\tilde{u}_2(A_1(x),A_2(x))=u_2(C_1(x),C_2(x))$. Inductively, we obtain easily for each $n\leq N$ a $\tilde{u}_n$ s.t.\ $\tilde{u}_n(A_1(x),\dots,A_n(x))=u_n(C_1(x),\dots,C_n(x))$. Hence the first term on the r.h.s.\ of (\ref{ineq:concave val fun}) is bounded from below by $v([\mu,\nu]_t)$. Using similar arguments for the second term, we see that $v([\mu,\nu]_{at+bs})\geq av([\mu,\nu]_{t})+bv([\mu,\nu]_{s}) $ holds if $s,t<1$. The case $s=1$ or $t=1$ can be obtained by a limiting and a Komlos-Mazur type argument (with convex combinations; here the convexity assumption is used). A direct argument, inspired by \cite{PflugPichler} and relying in convexity too, is as follows:
\begin{multline*}\textstyle \int H\bigl(A_1(x),\dots,A_N(x),
u_1\bigl(C_1(x)\bigr),\dots,u_N\bigl(C_1(x),\dots,C_N(x)\bigr)\bigr)\mu(dx) \geq\\ \textstyle
\EXP^{\mu}\left[ H\Bigl(A_1,\dots,A_N,
\EXP^{\mu}[u_1(C_1)|A_1,\dots,A_N],\dots,\EXP^{\mu}[u_N(C_1,\dots,C_N)|A_1,\dots,A_N]\Bigr)\right ], 
\end{multline*}
by the convexity assumption and Jensen's inequality, so now observing that if e.g.\ $t=1$ then $A=T$, which is bicausal from $\mu$ to $\nu$, we get that $\EXP^{\mu}[u_n(C_1,\dots,C_n)|A_1,\dots,A_N]=\tilde{u}_n(A_1,\dots,A_n)$ for each $n$ and we conclude as before. The case $s=1$ is analogous. 

For the last statement we obtain by the concavity assumption that
$$\textstyle H([1-t]x+tT(x),u^*(x))\leq H(x,u^*(x))+t\xi(x)\cdot[T(x)-x],$$
where $u^*$ is any optimizer for $v(\mu)$, and $\xi(x)$ is any measurable selection of $x\mapsto\partial_x H(x,u^*(x))$ (the partial superdifferential w.r.t.\ the first variable). Such a selection exists by \cite[Theorem 14.56]{RockWetsBook}, for which $H$ must be a normal integrand, but this is true by \cite[Proposition 14.39]{RockWetsBook}. In particular for $t=1$ and integrating we get 
$$\textstyle \int H(T(x),u^*(x))\mu(dx)\leq v(\mu)+ \int \xi(x)\cdot[T(x)-x]\mu(dx) .$$
By the same arguments as before (the convexity assumption, the bicausality of $T$ and $T_*\mu=\nu$), the l.h.s.\ is an upper bound for $v(\nu)$. All in all, if we could find a measurable selector $\xi(\cdot)$ such that 
$$\textstyle \xi(x)\in argmin_{\xi \in \partial_x H(x,u^*(x)) }\left\{\xi\cdot[T(x)-x]\right\},$$
this would finish the proof. This follows from the measurable maximum theorem \cite[Theorem 18.19]{Aliprantis} after observing that $(x,\xi)\mapsto \xi\cdot[T(x)-x] $ is a Carath\'eodory function and that the correspondence $x\mapsto  \partial_x H(x,u^*(x))$ is nonempty compact-valued and weakly measurable (i.e.\ measurable in the sense of \cite[Theorem 14.56]{RockWetsBook}).
\end{proof}

%The previous proof establishes concavity of $v$ along lexicographic displacement interpolation on $[0,1)$ even without assuming convexity of $H$ in the optimization variable. On the other hand, under the convexity assumption, one always has for $t\in [0,1]$ that
%
%$$\textstyle v(\pi_t)\geq (1-t)v(\mu)+tv(\nu),$$
%
%where $\pi_t(dz)=[(x,y)\mapsto(1-t)x+ty]_{*}\pi(dz)$ and $\pi$ is any bicausal plan of $\mu$ into $\nu$, so $\pi_t$ interpolates between these marginals. This is necessary but not sufficient for the full concavity of $t\mapsto v(\pi_t)$, which as we have seen, we are only able to prove for the Knothe-Rosenblatt Monge case.\\

Let us now discuss the geometric interpretation that lexicographic displacement interpolation should have. We start from the observation that for costs of the type
$$c_p(x,y):=\sum_{i\leq N}|x_i-y_i|^p,$$
%
%\comment{Referee 2: does not think it is distance. {\color{blue} JB: I give references in footnote and respond in length to the referee}}
lexicographic displacement interpolation has ``constant speed'' w.r.t.\ the associated bicausal / nested distances\footnote{That this constitutes an actual distance was proved in the appendix of \cite{PflugPichler}.}, namely:
\begin{equation}\label{eq constant speed}\textstyle
\left(\inf_{\gamma\in\Pi_{bc}(\mu,[\mu,\nu]_t)}\int c_pd\gamma\right)^{1/p}= t \left(\inf_{\gamma\in\Pi_{bc}(\mu,\nu)}\int c_pd\gamma\right)^{1/p},
\end{equation}
whenever $\mu$ has atomless conditional distributions. Indeed, the map $x\mapsto (1-t)x+tT(x)$ by assumption pushes forward $\mu$ into $[\mu,\nu]_t$, is bicausal, and is further an increasing triangular transformation m(See previous section), so it is the Knothe-Rosenblatt rearrangement of $\mu$ into $[\mu,\nu]_t$ and therefore optimal for the l.h.s.\ above. Hence
\small
\begin{align*}\textstyle
\left(\inf_{\gamma\in\Pi_{bc}(\mu,[\mu,\nu]_t)}\int c_pd\gamma\right)^{1/p} &= \textstyle\int\sum_{i\leq N} |x_i - (1-t)x_i - tT^i(x_i;x_1,\dots,x_{i-1})|^p d\mu \\&=\textstyle t^p\int \sum_{i\leq N} |x_i - T^i(x_i;x_1,\dots,x_{i-1})|^p d\mu , 
\end{align*}
which is the r.h.s.\ of \eqref{eq constant speed}. The argument is of course reminiscent to the Brenier case. This suggests that for the case $p=2$ one would want to interpret the corresponding bicausal distance as a geodesic length over the space of probability measures (say absolutely continuous ones) when given a differentiable structure and corresponding metric. This is discussed in \cite[Chapter 8]{Villani} for the classical transport case, and no such thing has yet been accomplished for the present bicausal setting. %, the main reason being the lack of an analogous hydrodynamical formulation of the problem. 
 In a way, a first step in this direction was done by Mikami \cite{Mikami_KR}, where a Hamilton-Jacobi formulation for the dual of the bicausal problem was derived. \\

We finally give the missing proof of Theorem \ref{thm: funct nested}:

\begin{proof}[Proof of Theorem \ref{thm: funct nested}]
By $(EXP)$ and \cite[Theorem 22.10 + (22.16)]{Villani_Old} we have that $\mu^{x_1,\dots,x_{t-1}}$ satisfies the next $T_1$ functional inequality (for every Borel prob.\ measure $m$):
\begin{align}\textstyle
\mathcal{W}_1(\mu^{x_1,\dots,x_{t-1}},m) & \textstyle \leq \frac{\sqrt{2}}{a_t}\left( 1+\log\int e^{a_tx_t^2}\mu^{x_1,\dots,x_{t-1}}(dx_t) \right)^{1/2}\sqrt{Ent(m|\mu^{x_1,\dots,x_{t-1}})}\\
&\leq \textstyle \frac{\sqrt{2(1+\lambda_t)}}{a_t} \sqrt{Ent(m|\mu^{x_1,\dots,x_{t-1}})}.
\notag
\end{align}
Let us denote by $K_{t-1}$ the constant in the r.h.s.\ above. By triangle inequality and $(LIP)$:
\begin{align*}\textstyle
\mathcal{W}_1(\mu^{x_1,\dots,x_{t-1}},\nu^{y_1,\dots,y_{t-1}})& \textstyle \leq  \mathcal{W}_1(\mu^{x_1,\dots,x_{t-1}},\mu^{y_1,\dots,y_{t-1}}) + \mathcal{W}_1(\mu^{y_1,\dots,y_{t-1}},\nu^{y_1,\dots,y_{t-1}}) \\
&\textstyle \leq C \sum_{i<t}|x_i-y_i| +K_{t-1}\sqrt{Ent(\nu^{y_1,\dots,y_{t-1}}|\mu^{y_1,\dots,y_{t-1}})},
\end{align*}
for $\mu\otimes\nu$-a.e.\ $(x_1,\dots,x_{t-1},y_1,\dots,y_{t-1})$; we are entitled to do this since we may assume w.l.o.g.\ that $\nu\ll\mu$. Using \eqref{Dyn-Pbc} and applying the above computation recursively, one arrives at:
$$\textstyle \mathcal{W}_{1,bc}(\mu,\nu)\leq\int \nu(dy_1,\dots,dy_N)\sum\limits_{j<N}K_{N-j-1}(1+C)^j\sqrt{Ent(\nu^{y_1,\dots,y_{N-j-1}}|\mu^{y_1,\dots,y_{N-j-1}})},$$
so using Cauchy-Schwartz for the sum and for the integral we get 
$$\textstyle\mathcal{W}_{1,bc}(\mu,\nu)\leq \sqrt{A}\sqrt{B},$$
where 
$$\textstyle
B= \int \nu(dy_1,\dots,dy_N)\sum\limits_{j<N}Ent(\nu^{y_1,\dots,y_{N-j-1}}|\mu^{y_1,\dots,y_{N-j-1}}),
$$
which by the additive property of the entropy (e.g.\ \cite[Lemma 10.3]{Varadhan}) equals $Ent(\nu|\mu)$, and where
$$\textstyle
A=2\sum\limits_{j<N}(1+C)^{2j} \frac{(1+\lambda_{N-j})}{a^2_{N-j}}%+  \sum\limits_{j<N}(1+C)^{2j}\int \nu(dy_1,\dots,dy_{N-j-1})\log\int e^{ay_{N-j}^2}\mu^{y_1,\dots,y_{N-j-1}}(dy_{N-j})
$$ 
\vspace{-\belowdisplayskip}\[\]
\end{proof}

\section{Counterexamples}\label{Sec Examples}

Previously we have discussed that the values of the causal and bicausal problems coincide in the case when the starting measure $\mu$ is the product of its marginals and the cost function has a separable structure. The following two examples show that dropping either assumption causes this equality to break down.

\begin{example}
Here $\mu$ is the product of its marginals, but the cost function  $c(x_1,x_2,y_1,y_2) = \IND_{(x_1,x_2)\neq (y_1,y_2)}$ is non-separable. Take
$$ \textstyle \mu = 0.16\, \delta_{(1,1)} +  0.24\, \delta_{(1,-1)} + 0.24\, \delta_{(-1,1)} + 0.36\, \delta_{(-1,-1)}, $$
$$ \nu = 0.25\, \delta_{(1,1)} +  0.25\, \delta_{(1,-1)} + 0.25 \, \delta_{(-1,1)} + 0.25\, \delta_{(-1,-1)}. $$
Then an optimal causal transport plan is
\begin{align*}\textstyle
\gamma_{c} & = 0.16\, \gamma{((1,1);(1,1) )} \,\,+ 0.24 \, \gamma{((1,-1);(1,-1))}  +  0.03 \, \gamma{((-1,1);(1,1))}   \\
& + 0.01 \, \gamma{((-1,1);(1,-1))} \,\,+ 0.2 \,\,\,\, \gamma{((-1,1);(-1,1))} \, + 0.06 \, \gamma{((-1,-1);(1,1))} \\ &+ 0.05 \,\,\, \gamma{((-1,-1);(-1,1))} + 0.25 \, \gamma{((-1,-1);(-1,-1))},
\end{align*} 
giving the optimal causal value $0.15$, and an optimal bicausal one is
\begin{align*}\textstyle
\gamma_{bc} & = 0.16\, \gamma{((1,1);(1,1))} + 0.04 \, \gamma{((1,-1);(1,1))}  +  0.2 \, \gamma{((1,-1);(1,-1))}  \\
& + 0.04 \, \gamma{((-1,1);(1,1))}+ 0.2 \,\,\,\, \gamma{((-1,1);(-1,1))} + 0.01 \, \gamma{((-1,-1);(1,1))} \\
&  + 0.05 \, \gamma{((-1,-1);(1,-1))}+ 0.05 \, \gamma{((-1,-1);(-1,1))}+ 0.25 \, \gamma{((-1,-1);(-1,-1))},
\end{align*} 
giving the value $0.19$.
{The optimal dual variables that correspond to reverse causality constraints are nonzero in this example, conceptually supporting the gap in the optimal values.
}
\end{example}

\begin{example}\label{counerexample}  Consider a quadratic-separable cost function and define $\mu$, failing to be the product of its marginals, and $\nu$ as:
\begin{align*}\textstyle
\mu &= 0.18\, \delta_{(1,2)} +  0.24\, \delta_{(1,0)} + 0.18 \, \delta_{(1,-2)} + 0.08\, \delta_{(-1,2)} + 0.12\, \delta_{(-1,0)} + 0.2\, \delta_{(-1,-2)}, \\
\nu &= 0.1\, \delta_{(1,2)} +  0.26\, \delta_{(1,-2)} + 0.16 \, \delta_{(-1,2)} + 0.48\, \delta_{(-1,-2)}.
\end{align*} 
An optimal causal plan is
\begin{align*}\textstyle
\gamma_{c} & = 0.144\, \gamma{((1,0);(1,-2))} \,\,\,  + 0.008 \, \gamma{((1,0);(-1,2))} \, +  0.088 \, \gamma{((1,0);(-1,-2))}  \\
&+ 0.1 \, \gamma{((1,2);(1,2))} + 0.008 \, \gamma{((1,2);(1,-2))}\,\,\,\, + 0.072 \, \gamma{((1,2);(-1,2))}\,  \\
& + 0.108 \, \gamma{((1,-2);(1,-2))} + 0.072 \, \gamma{((1,-2);(-1,-2))}+ 0.12 \,\,\, \gamma{((-1,0);(-1,-2))} \\ & + 0.08 \, \gamma{((-1,2);(-1,2))} + 0.2 \, \gamma{((-1,-2);(-1,-2))} ,
\end{align*} 
giving the value $2.528$, and a bicausal optimal one
\begin{align*}\textstyle
\gamma_{bc} & = 0.144\, \gamma{((1,0);(1,-2))}   + 0.096 \, \gamma{((1,0);(-1,-2))}  +  0.1 \, \gamma{((1,2);(1,2))}\\
& + 0.008 \, \gamma{((1,2);(1,-2))} + 0.06 \, \gamma{((1,2);(-1,2))}\,\,\, + 0.012 \, \gamma{((1,2);(-1,-2))} \\
&  + 0.108 \, \gamma{((1,-2);(1,-2))} + 0.072 \, \gamma{((1,-2);(-1,-2))}+ 0.02 \, \gamma{((-1,0);(-1,2))}  \\ &  + 0.1 \, \gamma{((-1,0);(-1,-2))}+ 0.08 \, \gamma{((-1,2);(-1,2))}+ 0.2 \, \gamma{((-1,-2);(-1,-2))},
\end{align*} 
with value $2.72$.
\end{example}
The previous example also shows us that Condition \ref{Ruesch+indep cond} and so the monotone regression condition in \cite[Corollary 2]{Rueschendorf}, which holds here, is insufficient to guarantee the equality between (\ref{Pc}) and (\ref{Pbc}) even for separable costs. 
Finally, we present an example showing that there may easily be no causal Monge maps even if classical Monge maps do exist. 
\begin{example} In the case, when 
$$ \textstyle\mu = a_1 \, \delta_{(1,2)} +  a_2\, \delta_{(1,0)} + a_3 \, \delta_{(-1,0)} + a_4\, \delta_{(-1,-2)}, $$
$$ \nu = a_1\, \delta_{(1,2)} +  a_3 \, \delta_{(1,0)} + a_2 \, \delta_{(-1,0)} + a_4\, \delta_{(-1,-2)}, $$
where $a_i, i=1, \dots,4$ are positive numbers that sum up to one and $a_1 \neq a_4$, one can easily observe that there is no causal Monge map pushing forward the former into the latter; i.e.\ the mass must split. On the other hand, a non-causal Monge map is given by:
\begin{align*}\textstyle
T(1,2)=(1,2)\,\,,\,\, T(-1,-2)=(-1,-2)\,\,,\,\, T(1,0)=(-1,0)\,\,,\,\,T(-1,0)=(1,0)\\
\gamma^{T}  = a_1\, \gamma{((1,2);(1,2))}   + a_2 \, \gamma{((1,0);(-1,0))}  +  a_3 \, \gamma{((-1,0);(1,0))} + a_4\, \gamma{((-1,-2);(-1,-2))}.
\end{align*} 
\end{example}

\section*{Acknowledgments} We thank B.\ Acciaio, J.\ Fontbona, R.\ Lassalle, W.\ Schachermayer and J.\ Yang for valuable discussions.

\bibliographystyle{amsalpha}
\bibliography{biblio_causal_transport}

\end{document}